\documentclass[numbers]{elsarticle}

\usepackage[utf8]{inputenc}
\usepackage[T1]{fontenc}
\usepackage[english]{babel}

\biboptions{square,comma,sort}

\usepackage{bbm,amssymb,amsfonts,amsmath,amsopn,amsthm,enumerate}

\usepackage{calrsfs}
\DeclareMathAlphabet{\pazocal}{OMS}{zplm}{m}{n}

\usepackage{nicefrac}

\newcommand{\Z}{\mathbbm Z}
\newcommand{\N}{\mathbbm N}

\newcommand{\R}{\mathbbm R}
\newcommand{\Rn}{\mathbbm R^n_x}
\newcommand{\eps}{\varepsilon}
\newcommand{\la}{\left\langle}
\newcommand{\ra}{\right\rangle}
\newcommand{\ral}{\right\rangle_{L^2}}

\newcommand{\Lt}{{L^2}}
\newcommand{\Lto}{{L^2}}
\newcommand{\Dn}{\Delta_\nu}

\newtheorem{thm}{Theorem}[section]
\newtheorem{prop}[thm]{Proposition}
\newtheorem{lem}[thm]{Lemma}

\newdefinition{rem}[thm]{Remark}

\DeclareMathOperator{\supp}{supp}
\DeclareMathOperator{\real}{Re}
\DeclareMathOperator{\Lip}{Lip}
\DeclareMathOperator{\Log}{Log}
\DeclareMathOperator{\spec}{spec}

\numberwithin{equation}{section}

\author[add1,email1]{D. Del Santo}
\address[add1]{Dipartimento di Matematica e Geoscienze\\
Universit\`a degli Studi di Trieste\\
Via Valerio 12/1, I-34127 Trieste, Italy}

\author[add2,email2]{Ch. P. J\"ah}
\address[add2]{Institut f\"ur Angewandte Analysis\\
Fakult\"at f\"ur Mathematik und Informatik\\
Technische Universit\"at Bergakademie Freiberg\\
Pr\"uferstrasse 9, D-09596 Freiberg, Germany}

\author[add1,email3]{M. Prizzi}

\fntext[email1]{delsanto@units.it}
\fntext[email2]{christian.jaeh@math.tu-freiberg.de}
\fntext[email3]{mprizzi@units.it}

\begin{document}

\begin{frontmatter}

\title{Conditional stability for backward parabolic equations with $\Log\Lip_t \times \Lip_x$-coefficients}

\begin{abstract}
In this paper we present an improvement of [Math. Ann. {\bf 345} (2009), 213--243], where the authors proved a result concerning continuous dependence for backward parabolic operators whose coefficients are Log-Lipschitz in $t$ and $C^2$ in $x$. The $C^2$ regularity with respect to $x$ had to be assumed for technical reasons. Here we remove this assumption, replacing it with Lipschitz-continuity. The main tools in the proof are Littlewood-Paley theory and Bony's paraproduct as well as a result of Coifman and Meyer [Ast\'erisque {\bf 57}, 1978, Th. 35].
\end{abstract}

\begin{keyword}
Conditional stability, backward parabolic equation, low regularity coefficients, weighted energy estimate, Bony's paraproduct \\ \hfill \\
\MSC[2010] 35B30 \sep 34A12 \sep 35A02 
\end{keyword}

\end{frontmatter}

\setcounter{tocdepth}{2}

\section{Introduction}
In this paper we study the continuous dependence of solutions to the Cauchy problem for 
a backward parabolic operator, namely 
\begin{equation} 
\label{BWop}
Pu = \partial_t u + \sum\limits_{j,k=1}^n \partial_{x_j}(a_{jk}(t,x)\partial_{x_k}u) = 0
\end{equation}
on the strip $[0,T] \times \R^n_x$ 
with data
\begin{equation} 
\label{BWdata}u(0,x)= u_0(x)
\end{equation}
 in $\R^n_x$.
 The coefficients are supposed to be measurable, real valued and bounded. The matrix $(a_{jk})_{j,k=1,\dots, n}$ is symmetric and positive definite, i.e.
 there exists $\kappa>0$ such that
 \begin{equation} 
 \label{ellipticity}
\sum\limits_{j,k=1}^n a_{jk}(t,x)\xi_j\xi_k\geq  \kappa|\xi|^2,
\end{equation}
for all $\xi\in \R^n_\xi$.

It is well known that the Cauchy problem (\ref{BWop}), (\ref{BWdata}) it is not well--posed in the sense of Hadamard \cite{Hadamard1952}, \cite{Hadamard1964}. On the one side the smoothing effect of the parabolic operator prevent to have existence results in any reasonable function space and on the other side (relatively) elementary examples show that also uniqueness is not valid without additional assumptions on the solutions and on the operator (see \cite{Tychonov}; for a more precise discussion on uniqueness of the solutions to the Cauchy problem for a backward parabolic equation we quote the papers \cite{Miller1973}, \cite{Mandache}, \cite{DSP2005}, \cite{DSP2012}, \cite{DSJP2013}).

In the celebrated paper \cite{John1960},  John introduced the notion of well--behaved problem in which also not well--posed problem can be included: roughly speaking a problem is well-behaved if its solutions in a space $\pazocal H$ depend 
continuously on the data belonging to a space $\pazocal K$, provided they satisfy a prescribed bound in possibly another space  $\pazocal H'$.

The well-behavedness for (\ref{BWop}), (\ref{BWdata}) in the space 
\begin{equation}
\label{H}
\pazocal H=C^0([0,T],L^2(\R^n_x))\cap C^0([0,T),H^1(\R^n_x))\cap C^1([0,T),L^2(\R^n_x))
\end{equation} 
with continuous dependence with respect to the data in $L^2(\R^n_x)$, can be deduced from the so called logarithmic convexity of the norm of the solutions to (\ref{BWop}),
as proved by Agmon and Nirenberg in \cite{AN1963}. A similar result was obtained by Glagoleva in \cite{Glagoleva1963} and in a more precise and general form by Hurd in \cite{Hurd1967}. Hurd's result can be summarized as follows:

\smallskip
\textit{suppose that the coefficients $a_{jk}$ are Lipschitz--continuous; for every $T' \in (0,T)$ and $D>0$, there exist $\rho>0$, $\delta \in (0,1)$ and $M >0$ such that if $u \in \pazocal H$ ($\pazocal H$ defined in (\ref{H})),  is a solution of $Pu \equiv 0$ on $[0,T]\times \R^n_x$ with $\|u(0,\cdot)\|_{L^2} \leq \rho$ and $\|u(T,\cdot)\|_{L^2} \leq D$, then \begin{equation} \label{illustration3}
\sup\limits_{t \in [0,T']} \|u(t,\cdot)\|_{\Lt} \leq M \|u(0,\cdot)\|_{\Lt}^\delta,
\end{equation} where the constants $\rho$, $M$ and $\delta$ depend only on $T'$, $D$, the ellipticity constant of $P$ and the Lipschitz constant of the coefficients with respect to $t$.}
\smallskip

Hurd's proof relies on some (complicate) energy estimates and it turns out that Lipschitz--continuity of the coefficients $a_{jk}$ is an essential requirement.

In the present paper we are interested in relaxing the regularity hypothesis on the coefficients $a_{jk}$. Our starting point are the results contained in  \cite{DSP2009}. In that paper  an example showed that if the coefficients $a_{jk}$ are not Lipschitz--continuous in time then the estimate (\ref{illustration3}) does not hold in general, and if the coefficients are $\log$--Lipschitz--continuous in time then a condition weaker than  (\ref{illustration3}) is valid. However, in order to obtain this weaker estimate, a technical difficulty imposed to assume the $C^2$--regularity with respect to the space variables. Here we  overcome this point and we remove this supplementary and not natural requirement. Our result is the following:

\smallskip
\textit{suppose that the coefficients $a_{jk}$ are Lipschitz--continuous with respect to  $x$ and log--Lipschitz--continuous with respect to $t$; for every $T' \in (0,T)$, $D>0$ and $s\in (0,1)$, there exist $\rho > 0$, $\delta \in (0,1)$ and $M$, $N>0$ such that if $u \in \pazocal H$ is a solution of $Pu \equiv 0$ on $[0,T] \times \R^n_x$ with $\|u(0,\cdot)\|_{L^2} \leq \rho$ and $\|u(T,\cdot)\|_{L^2} \leq D$,  then \begin{equation*}
\sup\limits_{t \in [0,T']} \|u(t,\cdot)\|_{L^2} \leq M \exp\left(-N|\log(\|u(0,\cdot)\|_{H^{-s}})|^\delta\right),
\end{equation*} where the constants $\rho$, $M$, $N$ and $\delta$ depend only on $T'$, $D$, $s$, the ellipticity constant of $P$, the Lipschitz constant of the coefficients with respect to $x$ and  the Log-Lipschitz constant of the coefficients with respect to $t$.}
\smallskip

The main tool in proving this statement is Bony's paraproduct (see \cite{Metivier}) and a theorem by Coifman and Meyer \cite[Th. 35]{CM1978}, which makes the estimate of a commutator more effective.

\paragraph{Outline of the content} In Sub--Section \ref{subsec:MainRes}, we state our main theorems and make some remarks regarding the comparison with the results of \cite{DSP2009}.

In Sub--Section \ref{sec:LPBPP}, we present elements of the Littlewood-Paley theory and we develop the necessary machinery of Bony's paraproduct for our proof. After that we proof auxiliary estimates that will be crucial for the proof of our weighted energy estimate in Sub--Sections \ref{sec:Aux1} and \ref{sec:Aux2}. Some proofs are shifted to the appendix in order to make the main results easier to read.

In Section \ref{sec:MainProof}, we prove a weighted energy estimate for solutions of \eqref{BWeq} from which the stability result Theorem \ref{thm:Stab} follows. The derivation of the stability result from the weighted energy estimate is shown in Section \ref{sec:Derivation}.

\section{Results} \label{sec:MainRes}               

\subsection{Notation}

We consider the backward-parabolic equation \begin{equation} \label{BWeq}
Pu = \partial_t u + \sum\limits_{j,k=1}^n \partial_{x_j}(a_{jk}(t,x)\partial_{x_k}u) = 0
\end{equation} on the strip $[0,T] \times \R^n_x$. We suppose that \begin{itemize}
\item for all $(t,x) \in [0,T] \times \R^n_x$ and for all $j,k = 1, \dots, n$, \begin{equation*}
a_{jk}(t,x) = a_{kj}(t,x);
\end{equation*}
\item there exists an $\kappa \in (0,1)$ such that for all $(t,x,\xi) \in [0,T] \times \R^n_x \times \R^n_\xi$, \begin{equation}
\kappa |\xi|^2 \leq \sum\limits_{j,k=1}^n a_{jk}(t,x)\xi_j \xi_k \leq \frac{1}{\kappa} |\xi|^2;
\end{equation}
\item for all $j,k=1,\dots,n$, $a_{jk} \in \Log\Lip([0,T],L^\infty(\R^n_x)) \cap L^\infty([0,T], \Lip(\R^n_x))$. 
\end{itemize}
We set \begin{eqnarray*}
&& A_{LL} := \sup\Big\{ \frac{|a_{jk}(t,x)-a_{jk}(s,x)|}{|t-s||1+\log|t-s||} : j,k = 1, \dots, n, \\
&& \qquad \qquad \qquad\qquad\qquad\qquad \qquad (t,s,x) \in [0,T]^2 \times \R^n_x, \, 0 < |s-t| \leq T  \Big\}, \\[0.3 cm]
&& A := \sup\{ \|\partial_x^\alpha a_{jk}(t,\cdot)\|_{L^\infty} : |\alpha| \leq 1, \, t \in [0,T] \}.
\end{eqnarray*}

\begin{rem} \label{rem:LowerOrder} If one would like to include lower order terms in \eqref{BWeq}, one has to suppose that those are $L^\infty$ with respect $t$ and also $\Lip$ with respect to $x$. The constants will then additionally depend on constants $B$ and $C$ similary defined to $A$. \end{rem}

\begin{rem} We will often use a letter, say $C$, to denote a generic numerical constant; and different appearances of the letter $C$ will not necessarily denote the same numerical constant, even in the same line of text. When a constant actually depends on one of the parameters of the problem, it shall be indicate by an index. Sometimes it might be necessary to differentiate between constants so that we will count them with an upper index.
\end{rem}

\subsection{Main results - stability and weighted energy estimates} \label{subsec:MainRes} 

We denote by \begin{equation*}
\pazocal H := C^0([0,T], L^2(\R^n_x)) \cap C^0([0,T),H^{1}(\R^n_x)) \cap C^1([0,T), L^2(\R^n_x))
\end{equation*} the space of solutions of \eqref{BWeq} for which we prove the stability result.

First we restate the precise local stability result of \cite{DSP2009}; we also want to compare the two estimates in the sequel. Keep in mind that in this case the constant $A$ also contains the $L^\infty$ norms of the second spatial derivative of the principal part coefficients.

\begin{thm}{\textbf{(Th. 1 in \cite{DSP2009})}} \label{thm:DSP2009} There exist a positive constant $\alpha_1$ and, setting $\sigma := \min\{T,\frac{1}{\alpha_1}\}$, 
$\bar{\sigma} = \frac{\sigma}{8}$, there exist constants $\rho$, $\delta$, $M$ and $N$, such that, whenever $u \in \pazocal H$ is a solution of \eqref{BWeq} with $\|u(0,\cdot)\|_{L^2} \leq \rho$, the inequality 
$$\sup\limits_{t \in [0,\bar{\sigma}]} \|u(t,\cdot)\|_{L^2} 
\leq M(1+\|u(\sigma, \cdot)\|_{L^2}) \exp(-N(|\log(\|u(0,\cdot)\|_\Lt)|^\delta)
$$
holds true.

The constant $\alpha_1$ depends only on $A_{LL}$, $A$, $\kappa$ and $n$, while the constants $\rho$, $\delta$, $M$ and $N$
depend on $A_{LL}$, $A$, $\kappa$, $n$ and $T$.
\end{thm}

Let us stress again that the constants $\alpha_1$, $\rho$,  $\delta$, $M$, $N$ depend also on similar constants $B$ and $C$ if one considers also lower order terms. See Remark \ref{rem:LowerOrder}.

The next results improves Theorem \ref{thm:DSP2009}: now the principal part coefficients are only Lipschitz continuous with respect to $x$.

\begin{thm}{\textbf{Conditional stability (local).}} \label{thm:Stab} Let $s \in (0,1)$. There exist a positive constant 
$\alpha_1$ and, setting $\sigma := \min\{T,\frac{1-s}{\alpha_1}\}$, 
$\bar{\sigma} = \frac{\sigma}{8}$, there exist constants $\rho$, $\delta$, $M$ and $N$, such that, whenever $u \in \pazocal H$ is a solution of \eqref{BWeq} with $\|u(0,\cdot)\|_{H^{-s}} \leq \rho$, the inequality 
 \begin{eqnarray} \label{est:Stab} \begin{aligned}
&\sup\limits_{t \in [0,\bar{\sigma}]} \|u(t,\cdot)\|_{L^2} \\
&\qquad \leq M\Big(1+\frac{1}{\sigma}\max\limits_{t \in [\frac{5}{8}\sigma,\frac{7}{8}\sigma]}\|u(t,\cdot)\|_\Lt\Big) \exp(-N(|\log(\|u(0,\cdot)\|_{H^{-s}})|^\delta)
\end{aligned}\end{eqnarray} holds true.

The constant $\alpha_1$ depends only on $A_{LL}$, $A$, $\kappa$, $s$ and $n$, while the constants $\rho$, $\delta$, $M$ and $N$
depend on $A_{LL}$, $A$, $\kappa$, $s$, $n$ and $T$.
\end{thm}

Iterating the local result of Theorem \ref{thm:Stab} a finite number of times, one obtains the following global continuous dependence result.

\begin{thm}{\textbf{Conditional stability (global).}} \label{thm:StabGlob} Let $s \in (0,1)$. Then, for $T' \in (0,T)$ and $D>0$ there exist positive constants$\rho'$, $\delta'$, $M'$ and $N'$, depending only on $A_{LL}$, $A$, $\kappa$, $n$, $s$ and $T'$ such that if $u \in \pazocal H$ is a solution of \eqref{BWeq} satisfying $\|u(0,\cdot)\|_{H^{-s}} \leq \rho$ and $\sup_{t\in 0,T}\|u(t,\cdot)\|_{L^2} \leq D$, the inequality \begin{eqnarray*}
\sup\limits_{t \in [0,T']} \|u(t,\cdot)\|_{\Lt} \leq M' \exp\big(-N'|\log(\|u(0,\cdot)\|_{H^{-s}})|^{\delta'}\big)
\end{eqnarray*} holds true. \end{thm}

\begin{rem} Theorems \ref{thm:Stab} and \ref{thm:StabGlob} hold also if one considers equation \eqref{BWeq} with lower order terms. As already mentioned, one has to assume Lipschitz-regularity in $x$ and the additional dependence of the constants on the $L^\infty$-norm and the $\Lip$-norm of those coefficients. \end{rem}

\subsubsection{Weighted energy estimates} \label{sec:WeightedEE}

The proof of Theorem \ref{thm:Stab} relies on an appropriate weighted energy estimate. The choice of the weight function is connected with the modulus of continuity with respect to $t$ as in \cite{DSP2009}. A similar situation occurred in \cite{DSP2005}, \cite{DSP2012}, where backward-uniqueness for parabolic operators by means of suitable Carleman estimates was obtained. In both cases, the weight function was deduced as a solution of a second order non-linear ordinary differential equation.

Let us now introduce the weight function that we are going to use here. For $s>0$, let $\mu(s)=s(1+|\log(s)|)$. For $\tau \geq 1$, we define \begin{equation*}
\theta(\tau) := \int_{\frac{1}{\tau}}^1 \frac{1}{\mu(s)} ds = \log(1+|\log(\tau)|).
\end{equation*} The function $\theta : [1,+\infty) \rightarrow [0,+\infty)$ is bijective and strictly increasing. For $y \in (0,1]$ and $\lambda > 1$, we set $\psi_\lambda(y) = \theta^{-1}(-\lambda \log(y)) = \exp(y^{-\lambda}-1)$ and we define \begin{equation*}
\Phi_\lambda(y) := -\int_y^1 \psi_\lambda(z) dz.
\end{equation*} The function $\Phi_\lambda :(0,1] \rightarrow (-\infty,0]$ is bijective and strictly increasing; moreover, it satisfies \begin{equation} \label{eq:DefPhi}
y\Phi''_{\lambda}(y)= -\lambda(\Phi_\lambda'(y))^2 \mu\big( \frac{1}{\Phi_\lambda'(y)} \big) =
-\lambda\Phi_\lambda'(y)\big( 1 + |\log\big(\frac{1}{\Phi_\lambda'(y)}\big)| \big).
\end{equation} This is the second order differential equation we mentioned above. The reason for this choice is made clear in \cite[Sec. 2]{DSP2009}. The computations in \cite{DSP2005}, \cite{DSP2012} lead to a different differential equation and consequently to a different weight. 
In the next lemma, we collect some properties of the functions $\psi_\lambda$ and $\Phi_\lambda$. The proof is left to the reader.

\begin{lem}\label{weight-properties} Let $\zeta>1$. Then, for $y\leq 1/\zeta$, \begin{equation*}
\psi_\lambda(\zeta y) = \exp(\zeta^{-\lambda}-1)(\psi_\lambda(y))^{\zeta^{-\lambda}}.
\end{equation*} Define $\Lambda_\lambda(y) := y \Phi_\lambda(1/y)$. Then the function $\Lambda_\lambda : [1,+\infty) \rightarrow (-\infty,0]$ is bijective and \begin{equation*}
\lim_{z \rightarrow -\infty} -\frac{1}{z}\psi_\lambda\big( \frac{1}{\Lambda_\lambda^{-1}(z)} \big) = +\infty.
\end{equation*}
\end{lem}

With these preparations, we are ready to state the energy estimate which will be needed to prove Theorem \ref{thm:Stab}.

\begin{prop}{\textbf{Weighted energy estimate.}} \label{thm:Energy} Let $s \in (0,1)$. Then, there exist positive constants $\bar{\lambda}>1$, $\bar{\gamma} $, $\alpha_1$ and $M>0$ such that, setting $\alpha := \max\{\alpha_1, T^{-1}\}$, $\sigma := \frac{1-s}{\alpha}$, $\tau  := \frac{\sigma}{4}$, letting $\beta \geq \sigma + \tau$ be a free parameter, then, for $u \in \pazocal H$ a solution of equation \eqref{BWeq}, one has \begin{eqnarray} \label{MainIn}
\nonumber && \int_0^p e^{2\gamma t}e^{-2\beta \Phi_\lambda\left( \frac{t+\tau}{\beta} \right)} \|u(t,\cdot)\|_{H^{1-s-\alpha t}}^2 dt \\
\nonumber && \qquad \qquad \leq M \Big( (p+\tau)e^{2\gamma p}e^{-2\beta \Phi_\lambda\left( \frac{p+\tau}{\beta} \right)} \|u(p,\cdot)\|_{H^{1-s-\alpha p}}^2 \\
\label{EnergyEst1} && \qquad \qquad \qquad \qquad + \tau \Phi'_\lambda\left( \frac{\tau}{\beta} \right) e^{-2\beta \Phi_\lambda\left( \frac{\tau}{\beta} \right)} \|u(0,\cdot)\|_{H^{-s}}^2 \Big)
\end{eqnarray} for all $p \in [0,\frac{7}{8}\sigma]$, $\lambda \geq \bar{\lambda}$ and $\gamma \geq \bar{\gamma}$.
The constant $\alpha_1$ depends only on $A_{LL}$, $A$, $\kappa$, $s$ and $n$, while the constants $\bar\lambda$, $\bar\gamma$  and $M$ 
depend on $A_{LL}$, $A$, $\kappa$, $s$, $n$ and $T$.

\end{prop}

Notice that this energy inequality undergoes a loss of derivatives. This phenomenon also occurred in \cite{CL1995}, \cite{CM2008} in the context of hyperbolic equations with Log-Lipschitz coefficients.

\section{Littlewood-Paley theory and Bony's paraproduct} 

In this section, we review some elements of the Littlewood-Paley decomposition which we shall use throughout this paper to define Bony's paraproduct. The proofs which are not contained in this section can be found in \cite{DSP2009}, \cite{DSP2012} and also \cite{Metivier}.

\subsection{Littlewood-Paley decomposition} \label{sec:LPBPP}

Let $\chi \in C_0^\infty(\R)$ with $0 \leq \chi(s) \leq 1$ be an even function and such that $\chi(s)=1$ for $|s| \leq 11/10$ and $\chi(s)=0$ for $|s| \geq 19/10$. We now define $\chi_k(\xi) = \chi(2^{-k}|\xi|)$ for $k \in \Z$ and $\xi \in \R^n_\xi$. Denoting by $\pazocal F$ the Fourier-transform $x \rightarrow \xi$ and by $\pazocal F^{-1}$ its inverse, we define the operators \begin{eqnarray*}
S_{-1}u = 0 &\text{and}&  S_ku = \chi_k(D_x)u = \pazocal F^{-1}(\chi_k(\cdot)\pazocal F(u)(\cdot)), \, k \geq 0, \\
\Delta_0 u = S_0 u &\text{and}& \Delta_k u = S_k u-S_{k-1}u, \, k \geq 1.
\end{eqnarray*} 
We define 
\begin{equation*}
\spec(u) := \supp(\pazocal F(u))
\end{equation*} and we will use the abbreviation $\Delta_k u = u_k$. For $u \in \pazocal S'(\Rn)$, we have \begin{eqnarray*}
u = \lim\limits_{k \rightarrow +\infty} S_k u = \sum\limits_{k \geq 0} \Delta_k u
\end{eqnarray*} in the sense of $\pazocal S'(\R^n_x)$.

We shall make use of the classical 
\begin{prop}[Bernstein's inequalities] \label{PropBernstein} 
Let $u \in \pazocal S'(\R^n_x)$. Then, for $\nu\geq 1$, 
\begin{equation} \label{bern1}
2^{\nu-1}\|u_\nu\|_{\Lt} \leq \|\nabla_x u_\nu\|_{L^2} \leq 2^{\nu+1} \|u_\nu\|_{\Lt}.
\end{equation} The right inequality of \eqref{bern1} holds also for $\nu = 0$.
\end{prop}

In the following two propositions, we recall the characterization of the classical Sobolev spaces and Lipschitz-continuous functions via Littlewood-Paley decomposition.

\begin{prop} \label{SobolevLP} Let $s \in \R$. Then, a tempered distribution $u \in \pazocal S'(\Rn)$ belongs to $H^s(\Rn)$ iff the following two conditions hold: \begin{enumerate}[(i)]
\item for all $k \geq 0$, $\Delta_k u \in L^2(\Rn)$
\item the sequence $\{\delta_k\}_{k \in \N}$, where $\delta_k := 2^{ks}\|\Delta_k u\|_{L^2}$, belongs to $l^2(\N)$.
\end{enumerate} Moreover, there exists $C_s \geq 1$ such that for all $u \in H^s(\R^n_x)$, we have \begin{equation*}
\frac{1}{C_s}\|u\|_{H^s} \leq \| \{ \delta_k \}_k \|_{l^2} \leq C_s \|u\|_{H^s}.
\end{equation*}
\end{prop}

\begin{prop} \label{RevSobolevLP}  Let $s \in \R$ and $R > 2$. If a sequence $\{u_k\}_{k \in \N} \subset L^2(\Rn)$ satisfies \begin{enumerate}[(i)]
\item $\spec(u_{0}) \subseteq \{|\xi| \leq R \}$ and $\spec(u_k) \subseteq \{R^{-1} 2^{k} \leq |\xi| \leq 2R 2^{k}\}$ for all $k \geq 1$ and 
\item the sequence $\{\delta_k\}_{k \in \N}$, where $\delta_k := 2^{ks}\|u_k\|_{L^2}$, belongs to $l^2(\N)$
\end{enumerate} then $u = \sum\limits_{k \geq 0} u_k \in H^s(\Rn)$ and there exists $C_s \geq 1$ such that we have \begin{equation*}
\frac{1}{C_s}\|u\|_{H^s} \leq \| \{ \delta_k\}_k \|_{l^2} \leq C_s \|u\|_{H^s}.
\end{equation*} When $s > 0$, it is enough to assume that for all $k \geq 0$, \begin{equation*}
\spec(u_k) \subseteq \{|\xi| \leq R 2^k \}
\end{equation*} holds true. \end{prop}

\begin{prop} \label{CharLip} A function $a \in L^\infty(\R^n_x)$ belongs to $\Lip(\R^n_x)$ iff \begin{equation*}
\sup\limits_{k \in \N_0} \|\nabla_x (S_k a)\|_{L^\infty} < +\infty.
\end{equation*} Moreover, there exists a positive constant $C$ such that if $a \in \Lip(\R^n_x)$, then \begin{equation*}
\|\Delta_k a\|_{L^\infty} \leq C 2^{-k} \|a\|_{\Lip}, \quad \text{and} \quad \|\nabla_x (S_k a)\|_{L^\infty} \leq C \|a\|_{\Lip}. \end{equation*} \end{prop}

\subsection{Bony's (modified) paraproduct} \label{LPBPP}

Let $a \in L^\infty(\R^n_x)$. Then, Bony's paraproduct of $a$ and $u \in H^s(\R^n_x)$ is defined as \begin{equation*}
T_a u = \sum\limits_{k \geq 3} S_{k-3}a \Delta_k u.
\end{equation*} For the proof of our continuous dependence result it is essential that $T_a$ is a positive operator. Unfortunately, this is not implied by $a(x) \geq \kappa > 0$. Therefore, we have to modify the paraproduct a little bit. We introduce the operator \begin{equation} \label{PPDef}
T_a^m u = S_{m-1}aS_{m+2}u + \sum\limits_{k \geq m+3} S_{k-3}a\Delta_k u,
\end{equation} where $m \in \N_0$; note $T^0_a = T_a$. As shall be shown in one of the subsequent propositions, the operator $T_a^m$ is a positive operator for positive $a$ provided that $m$ is sufficiently large. The  proofs of the subsequent propositions can be found in \cite{DSP2012}. We give proofs only if there are important points to them which are not contained in \cite{DSP2012}.

\begin{prop} \label{MappingProp} Let $m \in \N$, $s \in \R$ and $a \in L^\infty(\R^n_x)$. Then, $T_a^m$ maps $H^{s}(\R^n_x)$ continuously into $H^{s}(\R^n_x)$, i.e. there exists a constant $C_{m,s} >0$ such that \begin{equation*}
\|T_a^m u\|_{H^s} \leq C_{m,s} \|a\|_{L^\infty}\|u\|_{H^s}.
\end{equation*} If $m \in \N_{\geq3}$, $s \in (0,1)$ and $a \in L^\infty(\R^n_x) \cap \Lip(\R^n_x)$ then $a-T_a^m$ maps $H^{-s}(\R^n_x)$ continuously into $H^{1-s}(\R^n_x)$, i.e. there exists a constant $C_{m,s}>0$ such that \begin{equation*}
\|au-T_a^mu\|_{H^{1-s}} \leq C_{m,s}\|a\|_{\Lip}\|u\|_{H^{-s}}.
\end{equation*}
$C_{m,s}$ is independent of $s$ if $s$ is chosen in a compact subset of $(0,1)$.
\end{prop}

Next we state a positivity result for $T_a^m$.

\begin{prop} \label{prop:Pos} Let $a \in L^\infty(\R^n_x) \cap \Lip(\R^n_x)$ and suppose that $a(x) \geq \kappa >0$ for all $x \in \R^n_x$. Then, there exists a constant $m_0 = m_0(\kappa, \|a\|_{\Lip})$ such that \begin{equation*}
\real \la T_a^m u | u \ra_{L^2} \geq \frac{\kappa}{2}\|u\|_{L^2},
\end{equation*} for all $u \in L^2(\R^n_x)$ and $m \geq m_0$. A similar result is true for vector-valued functions if $a$ is replaced by a positive symmetric matrix.
\end{prop}


The next proposition is needed since $T_a^m$ is not self-adjoint. However, the operator $(T_a^m-(T_a^m)^\ast) \partial_{x_j}$ is of order $0$ and maps, if $a$ is Lipschitz, $L^2$ continuously into $L^2$.  

\begin{prop} \label{Adjoint} Let $m \in \N$, $a \in L^\infty(\R^n_x) \cap \Lip(\R^n_x)$ and $u \in L^2(\R^n_x)$. Then, there exists a constant $C_{m} > 0$ such that \begin{equation*}
\|(T_a^m-(T_a^m)^\ast) \partial_{x_j}u\|_\Lt \leq C_{m}\|a\|_{\Lip(\R^n)}\|u\|_{L^2}.
\end{equation*} \end{prop}

\subsection{Auxiliary estimates for $a-T_a^m$} \label{sec:Aux1}

Let $m\geq 3$. We set \begin{equation} \label{eq:AuxDef1}
(a-T_a^m)w = \sum\limits_{k \geq m} \Delta_ka S_{k-3}w + \sum\limits_{k \geq m}\big( \sum\limits_{\substack{j \geq 0 \\ |j-k| \leq 2}} \Delta_k a \Delta_j w \big) = \Omega_1 w + \Omega_2 w.
\end{equation} For our proof of the energy estimate from which we derive the stability result we need some estimates for terms involving $\Delta_\nu((a-T_a^m)w)$. To handle these terms, we introduce a second Littlewood-Paley decomposition depending on a parameter $\mu$ and we look at $\sum_{\mu \geq 0} \Delta_\nu((a-T_a^m)w_\mu)$. To derive estimates for those terms we need appropriate estimates for $\Delta_\nu\Omega_1 w_\mu$ and $\Delta_\nu \Omega_2 w_\mu$. Let us first analyze the spectra of $\Delta_\nu \Omega_1 w$ and $\Dn \Omega_2 w$: From the definition of $S_k$ and $\Delta_k$ in Section \ref{sec:LPBPP} we see that 
\begin{equation*}
\spec(\Delta_k a S_{k-3}w) \subseteq \{2^{k-2} \leq |\xi| \leq 2^{k+2} \}
\end{equation*} 
and therefore 
\begin{equation*}
\Dn \Omega_1 w = \sum\limits_{\substack{k \geq m \\ |k-\nu|\leq 2}} \Dn(\Delta_k a S_{k-3}w)
\end{equation*} since $\Dn (\Delta_k a S_{k-3}w) \equiv 0$ for $|\nu - k| \geq 3$. Replacing now $w$ by $w_\mu$ we get \begin{equation*}
\spec(S_{k-3}w_\mu) \subseteq \left\{ \begin{array}{ccl}
\emptyset &:& k \leq \mu +1, \\[0.3 cm]
\{2^{\mu-1} \leq |\xi| \leq 2^{\mu+1}\} &:& k \geq \mu+2,
\end{array} \right.
\end{equation*} and obtain from this 
\begin{equation*}
\spec(\Delta_k a S_{k-3}w_\mu) \subseteq \left\{ \begin{array}{ccl}
\emptyset &:& k \leq \mu + 1, \\[0.3 cm]
\{ |\xi| \leq 2^{k+2} \} &:& k = \mu+2,\\[0.3 cm]
\{2^{k-2} \leq |\xi| \leq 2^{k+2} \} &:& k \geq \mu+2.
\end{array} \right. 
\end{equation*} 
With this we get 
$$
\Dn \Omega_1 w_\mu = \sum\limits_{\substack{k \geq \max\{m,\,\mu+2 \}\\ |\nu -k|\leq 1}} \Dn (\Delta_k a S_{k-3}w_\mu).
$$
Further, we also get $\Dn \Omega_1 w_\mu \equiv 0$ for $\nu \leq \mu-1$. Now we look at $\Delta_\nu \Omega_2 w_\mu$. We have 
\begin{eqnarray*}
\Delta_\nu \Omega_2 w_\mu &=& \Delta_\nu\big(\sum\limits_{k \geq m}\big( \sum\limits_{\substack{j \geq 0 \\ |j-k| \leq 2}} \Delta_k a \Delta_j w_\mu \big)\big) \\
&=& \Delta_\nu\big( \sum\limits_{ |\mu-j| \leq 1} \;\sum\limits_{ |j-k|\leq 2} \Delta_k a \Delta_j w_\mu\big)
\end{eqnarray*} 
since 
\begin{equation*}
\spec(\Delta_j(\Delta_\mu w)) \subseteq \left\{ \begin{array}{ccl}
\emptyset &:& |j-\mu| \geq 2, \\[0.3 cm]
\{2^{\mu-1} \leq |\xi| \leq 2^{\mu+1}\} &:& |j-\mu| \leq 1.
\end{array} \right.
\end{equation*} 
From that we get \begin{equation} \label{aux:est15}
\Dn \Omega_2 w_\mu = \Delta_\nu \big(\sum\limits_{|\mu-j|\leq 1} \;\sum\limits_{\substack{k \geq m \\ |k-j| \leq 2}} \Delta_k a \Delta_j(\Delta_\mu w)\big)
\end{equation} with 
\begin{equation*}
\spec ( \Dn \Omega_2 w_\mu ) \subseteq \{|\xi| \leq 2^{\mu+5}\}, \quad \nu \leq \mu + 5.
\end{equation*} For all $\nu \geq \mu+6$ we have $\Delta_\nu \Omega_2 w_\mu \equiv 0$.

We prove now some technical lemmas which we will use later on.

\begin{lem} \label{lem:Omega1} Let $s' \in (0,1)$, $m \in \N_0$, $a \in L^\infty(\R^n_x) \cap \Lip(\R^n_x)$ and $w \in L^2(\Rn)$. Then, there exist a constant $C>0$ and a sequence $\{c_\nu^{(\mu)}\}_{\nu \in \N_0} \in l^2(\N_0)$, depending on $\Delta_\mu w$, with $\|\{c_\nu^{(\mu)}\}_\nu\|_{l^2} \leq 1$ for all $\mu \geq 0$, such that \begin{equation} \label{eq:Aux1}
\|\Dn \Omega_1 w_\mu\|_\Lt \leq C 2^{-\nu(1-s')} c_{\nu}^{(\mu)} \|w_\mu\|_{L^2}.
\end{equation} \end{lem}

\begin{proof} From our considerations above we have that \begin{equation*}
\Dn \Omega_1 w_\mu = \sum\limits_{\substack{ |k-\nu|\leq 2}} \Dn(\Delta_k a S_{k-3}w_\mu)
\end{equation*} and therefore \begin{eqnarray}
\nonumber \|\Dn(\Omega_1 w_\mu)\|_{\Lt} &\leq& \sum\limits_{k: |k-\nu| \leq 1} \|\Delta_k a S_{k-3}w_\mu\|_\Lt \\
\nonumber &\leq& \sum\limits_{|k-\nu| \leq 2} \|\Delta_k a\|_{L^\infty} \|S_{k-3}w_\mu\|_\Lt \\
\nonumber &\leq& C \sum\limits_{|k-\nu| \leq 2} \|a\|_{\Lip} 2^{-k} \sum\limits_{j \leq k} \|\Delta_j w_\mu\|_\Lt \\
\nonumber &=& C \|a\|_{\Lip} \sum\limits_{ |k-\nu| \leq 2} 2^{-k} \sum\limits_{j \leq k} 2^{ks'}2^{-ks'}2^{js'}\underbrace{2^{-js'}\|\Delta_j w_\mu\|_\Lt}_{=: \eps_j^{(\mu)}} \\
\nonumber &\leq& C \|a\|_{\Lip} \sum\limits_{ |k-\nu| \leq 2} 2^{-(1-s')k} \underbrace{\sum\limits_{j \leq k} 2^{-(k-j)s'}\eps_j^{(\mu)}}_{=: f_k^{(\mu)}} \\
\nonumber &=& C \sum\limits_{ |k-\nu| \leq 2} \|a\|_{\Lip} 2^{-(1-s')k} f_k^{(\mu)} \\
\label{eq:Aux2} &\leq& C \|a\|_{\Lip} 2^{-(1-s')\nu}  \sum\limits_{ |k-\nu| \leq 2}f_k^{(\mu)},
\end{eqnarray} where $\{\eps_j^{(\mu)}\}_{j \in \N_0} \in l^2(\N_0)$ with $\|\{\eps_j^{(\mu)}\}\|_{l^2} \approx \|w_\mu\|_{H^{-s'}}$; see Proposition \ref{SobolevLP}. The sequence $\{f_k^{(\mu)}\}_{k \in \N_0}$ is a convolution of the sequences $\{\eps_j^{(\mu)}\}_{j \in \N_0}$ and $d_k := 2^{-ks'}$. Using Young's inequality, we obtain \begin{equation*}
\|\{f_k^{(\mu)}\}_k\|_{l^2} = \|\{\{\eps_j^{(\mu)}\} \ast_{(j)} \{d_k\}\}_k\|_{l^2} \leq \|\{d_k\}_k\|_{l^1} \|\{\eps_j^{(\mu)}\}_j\|_{l^2}.
\end{equation*} From the formula of the geometric series and the integral criterion, we obtain \begin{equation*}
\|\{d_k\}_k\|_{l^1} \leq \frac{1}{1-2^{-s'}} \leq \frac{C}{s'}
\end{equation*} and hence, we get \begin{equation*}
\|\{f_k\}_k\|_{l^2} \leq \frac{C}{s'} \|w_\mu\|_{H^{-s'}}.
\end{equation*} We define 
\begin{equation*}
c_\nu := \frac{f_{\nu-2}^{(\mu)} + f_{\nu-1}^{(\mu)} + f_\nu^{(\mu)} + f_{\nu+1}^{(\mu)}+ f_{\nu+2}^{(\mu)}}{C_{s'}\|w_\mu\|_{H^{-s'}}},
\end{equation*} 
where $C_{s'}$ can be chosen such that $\sum_{\nu \geq 0} (c_\nu^{(\mu)})^2 \leq 1$. With this, we get from \eqref{eq:Aux2} \begin{equation*}
\|\Dn \Omega_1 w_\mu\|_\Lt \leq C \|a\|_{\Lip} 2^{-(1-s')\nu} c_{\nu}^{(\mu)} \|w_\mu\|_{H^{-s'}}.
\end{equation*} Using the embedding of $L^2$ into $H^{-s'}$, we finally obtain \eqref{eq:Aux1}. \end{proof}

%

The next lemma deals with the estimate of $\Delta_\nu \Omega_2 w$.

\begin{lem} \label{lem:Omega2} Let $s' \in (0,1)$, $m \in \N_0$, $a \in L^\infty(\R^n_x) \cap \Lip(\R^n_x)$ and $w \in L^2(\R^n_x)$. Then, there exist a constant $C >0$ and a sequence $\{\tilde{c}_\nu^{(\mu)}\}_{\nu \in \N_0}\ \in l^2(\N_0)$, depending on $\Delta_\mu w$, with $\|\{\tilde{c}_\nu^{(\mu)}\}_\nu\|_{l^2} \leq 1$ for all $\mu \geq 1$, such that \begin{equation*}
\|\Dn \Omega_2 w_\mu\|_\Lt \leq C \|a\|_{\Lip} \tilde{c}_\nu^{(\mu)} 2^{-\mu} \|w_\mu\|_{L^2}.
\end{equation*} \end{lem}

\begin{proof} Straightforward computations on \eqref{aux:est15} show that $\Omega_2w \in L^2(\R^n_x)$ if $w \in L^2(\R^n_x)$. Hence, there exists a sequence $\{c_\nu^{(\mu)}\}_{\nu \in \N_0}$, depending on $w_\mu$, with $\|\{c_\nu^{(\mu)}\}_\nu\|_{l^2} \approx \|\Omega_2 w_\mu\|_{L^2}$. From \eqref{aux:est15}, we obtain 
\begin{eqnarray*}
\|\Dn \Omega_2 w_\mu\|_\Lt &\leq&  \tilde{c}_\nu^{(\mu)}\|\Omega_2 w_\mu\|_\Lt \\
&\leq& \tilde{c}_\nu^{(\mu)}\sum\limits_{|\mu-j|\leq 1} \sum\limits_{\substack{k \geq m \\ |k-j| \leq 2}} \|\Delta_k a \Delta_j(\Delta_\mu w)\|_{L^2} \\
&\leq&  \tilde{c}_\nu^{(\mu)} \sum\limits_{|j-\mu|\leq 1} \sum\limits_{\substack{k \geq m \\ |j-k|\leq 2}} 2^{-k} \|a\|_{\Lip} \|w_\mu\|_\Lt \\
&\leq&  \|a\|_{\Lip} \tilde{c}_\nu^{(\mu)} 2^{-\mu} \|w_\mu\|_\Lt,
\end{eqnarray*} where $\tilde{c}_\nu^{(\mu)} = c_\nu^{(\mu)} / \|\Omega_2 w_\mu\|_{L^2}$. By construction we have $\sum_{\nu \geq 0} (\tilde{c}_\nu^{(\mu)})^2 \leq 1$ for all $\mu \geq 0$. This concludes the proof. \end{proof}

The next proposition is at the very heart of our proof and contains information about the behavior of the Littlewood-Paley pieces of $(a-T_a)w$.

\begin{prop} \label{prop:AuxP1} Let $s \in (0,1)$, $m \in \N_0$, $a \in L^\infty(\R^n_x)\cap \Lip(\R^n_x)$, $\alpha>0$ and $t \in [0,\frac{7}{8}\sigma]$, $\sigma := \frac{1-s}{\alpha}$. Then there exists a constant $C_{s}>0$ such that, for all $w \in \pazocal H$, we have \begin{eqnarray*}
&& \sum\limits_{\nu \geq 0} 2^{-(s+\alpha t)\nu}\la \partial_{x_i} \partial_t v_\nu(t,\cdot) | \Delta_\nu((a-T_a^m)\partial_{x_j} w(t,\cdot)) \ral \\
&& \qquad \qquad \leq \frac{1}{N} \sum\limits_{\nu \geq 0} \|\partial_t v_\nu(t,\cdot)\|_\Lt^2 + C_s N \sum\limits_{\nu \geq 0} 2^{2\nu}\|v_\nu(t,\cdot)\|_\Lt^2
\end{eqnarray*} for every $N >0$ and with $v_\nu = 2^{-(s+ \alpha t)\nu}w_\nu$. \end{prop}

The proof of this proposition can be found in the appendix. Following the same ideas one can also prove

\begin{prop} Let $s \in (0,1)$, $m \in \N_0$, $a \in L^\infty(\R^n_x)\cap \Lip(\R^n_x)$, $\alpha>0$ and $t \in [0,\frac{7}{8}\sigma]$, $\sigma := \frac{1-s}{\alpha}$. Then there exists a constant $C_{s}>0$ such that, for all $w \in \pazocal H$, we have \begin{eqnarray*}
\sum\limits_{\nu \geq 0} 2^{-(s+\alpha t)\nu} \nu \la \partial_{x_i} v_\nu(t,\cdot) | \Delta_\nu((a-T_a^m)\partial_{x_j} w(t,\cdot)) \ral
\leq C_s \sum\limits_{\nu \geq 0} 2^{2\nu}\|v_\nu(t,\cdot)\|_\Lt^2
\end{eqnarray*} for every $N>0$ and with $v_\nu = 2^{-(s+\alpha t)\nu}w_\nu$. \end{prop}

\subsection{Auxiliary estimates for $[\Dn,T_a^m]$} \label{sec:Aux2}

The next result about commutation will also be crucial in our proof of the energy estimate \eqref{EnergyEst1}. It also plays an essential role in the proof of Carleman estimates for \eqref{BWop} with low-regular coefficients \cite{DSP2012,DS2012} and in the well-posedness for hyperbolic equations with low-regular coefficients \cite{CM2008}.

\begin{prop} Let $m \in \N_{\geq 3}$, $a \in L^\infty(\R^n_x) \cap \Lip(\R^n_x)$ and $s \in (0,1)$. Then, for $t \in [0,\frac{7}{8}\sigma]$, $\sigma := \frac{1-s}{\alpha}$ there exists a constant $C_m >0$ such that for all $w \in \pazocal H$ \begin{eqnarray*}
&& \sum\limits_{\nu \geq 0} 2^{-2(s+\alpha t)\nu} \la \partial_t \partial_{x_j} v_\nu(t,\cdot) | [\Delta_\nu, T_a^m ] \partial_{x_h} w(t,\cdot) \ral \\
&& \qquad \qquad \leq \frac{1}{N}\sum\limits_{\nu \geq 0} \| \partial_t v_\nu(t,\cdot)\|_\Lt^2  + \frac{C_m}{1-s} \|a\|_{\Lip}N \sum\limits_{\nu \geq 0} 2^{2\nu}\|v_\nu(t,\cdot)\|_\Lt^2
\end{eqnarray*} for every $N > 0$ and with $v_\nu = 2^{-(s + \alpha t)\nu}w_\nu$. \end{prop}

This follows from the following lemma whose proof can be found in the appendix.

\begin{lem}\label{lemmacomm} Let $m \in \N_0$, $a \in L^\infty(\Rn) \cap \Lip(\Rn)$. Then there exists a constant $C_m >0$
such that, for all $w(t,\cdot) \in \pazocal H$,  
\begin{equation*} 
\sum\limits_{\nu \geq 0} 2^{-2(s+\alpha t)\nu}\Big\|\partial_{x_j} [\Delta_\nu, T_a^m ] \partial_{x_h} w(t,\cdot)\Big\|_\Lt^2 \leq \frac{C_m}{1-s} \|a\|_{\Lip}^2 \sum\limits_{\nu \geq 0} 2^{2\nu}\|v_\nu(t,\cdot)\|_{L^2}^2
\end{equation*} with $v_\nu = 2^{-(s + \alpha t)\nu}w_\nu$. \end{lem}

Also the next proposition follows immediately from this lemma.

\begin{prop} \label{Commutator} Let $m \in \N_{\geq 3}$, $a \in L^\infty(\R^n_x) \cap \Lip(\R^n_x)$ and $s \in (0,1)$. Then, for $t \in [0,\frac{7}{8}\sigma]$, $\sigma := \frac{1-s}{\alpha}$ there exists a constant $C_m >0$ such that for all $w \in H^{1-s-\alpha t}(\R^n_x)$ \begin{eqnarray*}
&& \sum\limits_{\nu \geq 0} 2^{-2(s+\alpha t)\nu} \nu \la \partial_{x_j} v_\nu(t,\cdot) | [\Delta_\nu, T_a^m ] \partial_{x_h} w(t,\cdot) \ral \\
&& \qquad \qquad \qquad \qquad  \leq \frac{C_m}{1-s} \|a\|_{\Lip} \sum\limits_{\nu \geq 0} 2^{2\nu}\|v_\nu(t,\cdot)\|_\Lt^2
\end{eqnarray*} for every $N > 0$ and with $v_\nu = 2^{-(s + \alpha t)\nu}w_\nu$. \end{prop}


\section{Proof of Proposition \ref{thm:Energy}} \label{sec:MainProof}

In order to simplify the presentation, we shall write the proof only for $n=1$. As already mentioned, one may also include lower-order terms with the appropriate regularity in $x$; see Section \ref{subsec:MainRes}. The latter can be handled with the techniques of the present work following the scheme of \cite{DSP2009}.

To make the proof more readable, we divide it into several steps. First the operator will be transformed by a change of variables involving the weight function, and then we shall introduce the paraproduct and microlocalize the operator. After that, we shall use the estimates of Section \ref{LPBPP} and conclude the proof for $\nu=0$ and $\nu \geq 1$ separately. After taht,  in Section \ref{sec:Derivation} we shall show how the stability estimate follows from the energy estimate.

\subsection{Preliminaries - transformation, microlocalization, approximation} \label{sec:Prel}

Let $u \in \pazocal H$ be a solution of the equation \begin{equation*}
Pu = \partial_t u + \partial_{x}(a(t,x)\partial_x u) = 0
\end{equation*} on the strip $[0,T] \times \R_x$.
In what follows, $\alpha_1>0$, $\bar\lambda>1$ and $\bar\gamma>0$ are constants to be determined later. Set $\alpha:=\max\{\alpha_1,T^{-1}\}$, take $s\in(0,1)$, and set $\sigma:=\frac{1-s}\alpha$, $\tau:=\frac\sigma4$. For $\gamma\geq\bar\gamma$, $\lambda\geq\bar\lambda$ and $\beta\geq\sigma+\tau$, define
$w(t,x) = e^{\gamma t}e^{-\beta\Phi_\lambda\left( \frac{t+\tau}{\beta} \right)}u(t,x)$. Then $w$ satisfies the following equation: \begin{equation*}
w_t -\gamma w + \Phi_\lambda'\left( \frac{t+\tau}{\beta} \right)w + \partial_x(a(t,x)\partial_x w) = 0.
\end{equation*}   Now we add and subtract $\partial_x T_a^m \partial_x w$, with $T_a^m$ as defined in \eqref{PPDef}, and obtain \begin{equation} \label{eq:P1}
w_t -\gamma w + \Phi_\lambda'\left( \frac{t+\tau}{\beta} \right)w + \partial_x(T_a^m\partial_x w) + \partial_x((a-T_a^m)\partial_x w) = 0.
\end{equation} We set $u_\nu = \Delta_\nu u$, $w_\nu = \Delta_\nu w$ and $v_\nu = 2^{-(s+\alpha t)\nu}w_\nu$. The function $v_\nu$ satisfies \begin{equation} \label{eq:P2} \begin{aligned}
& \partial_t v_\nu = \gamma v_\nu - \Phi_\lambda'\left( \frac{t+\tau}{\beta} \right)v_\nu - \partial_x(T_a^m\partial_x v_\nu) - \alpha\log(2) \nu v_\nu \\
& \qquad - 2^{-(s+\alpha t)\nu} \partial_x([\Dn, T_a^m]\partial_x w) - 2^{-(s+\alpha t)\nu} \Delta_\nu \partial_x((a-T_a^m)\partial_x w).
\end{aligned}\end{equation} Next, we compute the scalar product of \eqref{eq:P2} with $(t+\tau)\partial_t v_\nu$  and obtain \begin{equation} \label{eq:Mod} \begin{aligned}
& (t+\tau)\|\partial_t v_\nu(t,\cdot)\|_\Lto^2 = \gamma(t+\tau) \la v_\nu | \partial_t v_\nu(t,\cdot)\ral \\
& \qquad \quad -(t+\tau)\la \Phi_\lambda'\left( \frac{t+\tau}{\beta}\right)v_\nu(t,\cdot) | \partial_t v_\nu(t,\cdot) \ral \\
& \qquad \quad -(t+\tau)\la \partial_x(T_a^m\partial_x v_\nu(t,\cdot)) | \partial_t v_\nu(t,\cdot) \ral \\
& \qquad \quad -\alpha\log(2)(t+\tau)\nu \la v_\nu(t,\cdot) | \partial_t v_\nu(t,\cdot) \ral \\
& \qquad \quad -(t+\tau) 2^{-(s+\alpha t)\nu} \la \partial_x([\Dn, T_a^m]\partial_x w(t,\cdot)) | \partial_t v_\nu(t,\cdot) \ral \\
& \qquad \quad -(t+\tau)2^{-(s+\alpha t)\nu} \la \Delta_\nu \partial_x((a-T_a^m)\partial_x w(t,\cdot)) | \partial_t v_\nu(t,\cdot) \ral.
\end{aligned}\end{equation} To proceed, we have to regularize the coefficient $a(t,x)$ with respect to $t$. Therefore, we pick an even, non-negative $\rho \in C_0^\infty(\R)$ with $\supp(\rho) \subseteq [-\frac{1}{2},\frac{1}{2}]$ and $\int_\R \rho(s) ds = 1$. For $\eps \in (0,1]$, we set \begin{equation*}
a_\eps(t,x) = \frac{1}{\eps} \int_\R a(s,x)\rho\left( \frac{t-s}{\eps} \right) ds.
\end{equation*} A straightforward computation shows that for all $\eps \in (0,1]$, we have \begin{eqnarray}
&& a_\eps(t,x) \geq a_0 > 0 \\
&& |a_\eps(t,x)-a(t,x)| \leq A_{LL} \eps(|\log(\eps)|+1)
\end{eqnarray} as well as \begin{equation*}
|\partial_t a_\eps(t,x)| \leq A_{LL}\|\rho'\|_{L^1(\R)}(|\log(\eps)|+1)
\end{equation*} for all $(t,x) \in [0,T] \times \R_x$. From these properties of $a_\eps(t,x)$, the fact that $T_{a+b} = T_a + T_b$ and Proposition \ref{MappingProp}, we immediately get \begin{lem} \label{lem:Approx} Let $m \in \N_0$ and $u \in L^2(\R^n_x)$. Then \begin{eqnarray*}
&& \|(T_a^m-T_{a_\eps}^m)u\|_\Lto \leq C_m A_{LL}\eps(|\log(\eps)|+1) \|u\|_\Lto \,\, \text{and} \\
&& \|T_{\partial_t a_\eps}^m u\|_\Lto \leq C_m A_{LL} \|\rho'\|_{L^1(\R)}\|u\|_{\Lt}
\end{eqnarray*} hold. \end{lem}

We set \begin{equation*}
a_\nu(t,x) := a_\eps(t,x), \, \text{ with } \, \eps = 2^{-2\nu}.
\end{equation*}

We replace $T_a^m$ by $T_{a_\nu}^m + T_a^m-T_{a_\nu}^m$ in the third term of the right hand side of \eqref{eq:Mod} and we obtain \begin{equation} \label{eq:Mod1} \begin{aligned}
& (t+\tau)\|\partial_t v_\nu(t,\cdot)\|_\Lto^2 = \gamma(t+\tau) \la v_\nu(t,\cdot) | \partial_t v_\nu(t,\cdot)\ral \\
& \qquad \quad -(t+\tau)\la \Phi_\lambda'\left( \frac{t+\tau}{\beta}\right)v_\nu(t,\cdot) | \partial_t v_\nu(t,\cdot) \ral \\
& \qquad \quad -(t+\tau)\la \partial_x(T_{a_\nu}^m\partial_x v_\nu(t,\cdot)) | \partial_t v_\nu(t,\cdot) \ral  \\
& \qquad \quad -(t+\tau)\la \partial_x((T_a^m-T_{a_\nu}^m)\partial_x v_\nu(t,\cdot)) | \partial_t v_\nu(t,\cdot) \ral \\
& \qquad \quad -\alpha\log(2)(t+\tau) \nu \la v_\nu(t,\cdot) | \partial_t v_\nu(t,\cdot) \ral \\
& \qquad \quad -(t+\tau) 2^{-(s+\alpha t)\nu} \la \partial_x([\Dn, T_a^m]\partial_x w(t,\cdot)) | \partial_t v_\nu(t,\cdot) \ral \\
& \qquad \quad -(t+\tau)2^{-(s+\alpha t)\nu} \la \Delta_\nu \partial_x((a-T_a^m)\partial_x w(t,\cdot)) | \partial_t v_\nu(t,\cdot) \ral.
\end{aligned}\end{equation}

Now we replace $\partial_t v_\nu(t,\cdot)$ in \begin{equation*}
-\alpha\log(2)(t+\tau) \nu \la v_\nu(t,\cdot) | \partial_t v_\nu(t,\cdot) \ral
\end{equation*} by the expression in the right hand side of \eqref{eq:P2} and we obtain 

\begin{equation} \label{eq:Mod2} \begin{aligned}
& -\alpha \log(2)(t+\tau) \nu \la v_\nu(t,\cdot) | \partial_t v_\nu(t,\cdot) \ral = \\
&\quad -\alpha \gamma \log(2)\nu(t+\tau)\|v_\nu(t,\cdot)\|_\Lto^2 \\
&\quad +\alpha\log(2)(t+\tau)\Phi_\lambda'\left( \frac{t+\tau}{\beta} \right)\nu \|v_\nu(t,\cdot)\|_\Lto^2 \\
&\quad +\alpha\log(2)(t+\tau)\nu \la v_\nu(t,\cdot) | \partial_x T_a^m\partial_x v_\nu(t,\cdot) \ral \\
&\quad +\alpha^2(\log(2))^2 (t+\tau) \nu^2 \|v_\nu(t,\cdot)\|_\Lto^2 \\
&\quad +\alpha \log(2)\nu 2^{-(s+\alpha t)\nu} (t+\tau) \la v_\nu(t,\cdot) | \partial_x\big( [\Dn,T_a^m]\partial_x w(t,\cdot) \big) \ral \\
&\quad +\alpha \log(2)\nu 2^{-(s+\alpha t)\nu} (t+\tau) \la v_\nu(t,\cdot) | \Dn \partial_x\big( (a-T_a^m)\partial_x w(t,\cdot) \big) \ral.
\end{aligned}\end{equation}

With \eqref{eq:Mod1} and \eqref{eq:Mod2}, we obtain \begin{equation*}\begin{aligned}
&(t+\tau)\|\partial_t v_\nu(t,\cdot)\|_\Lto^2 = \\
&\qquad \quad \gamma(t+\tau) \la v_\nu(t,\cdot) | \partial_t v_\nu(t,\cdot) \ral \\
&\qquad \quad -(t+\tau) \Phi_\lambda'\left( \frac{t+\tau}{\beta}\right) \la v_\nu(t,\cdot) | \partial_t v_\nu(t,\cdot) \ral \\
&\qquad \quad -(t+\tau)\la \partial_x(T_{a_\nu}^m\partial_x v_\nu(t,\cdot)) | \partial_t v_\nu(t,\cdot) \ral \\
&\qquad \quad -(t+\tau)\la \partial_x((T_a^m-T_{a_\nu}^m)\partial_x v_\nu(t,\cdot)) | \partial_t v_\nu(t,\cdot) \ral \\
&\qquad \quad +\alpha\log(2)(t+\tau)\Phi_\lambda'\left( \frac{t+\tau}{\beta} \right)\nu \|v_\nu(t,\cdot)\|_\Lto^2 \\
&\qquad \quad +\alpha\log(2)(t+\tau)\nu \la v_\nu(t,\cdot) | \partial_x T_a^m\partial_x v_\nu(t,\cdot) \ral \\
&\qquad \quad +\alpha^2(\log(2))^2 (t+\tau) \nu^2 \|v_\nu(t,\cdot)\|_\Lto^2 \\
&\qquad \quad -\alpha\gamma \log(2)(t+\tau) \nu \|v_\nu(t,\cdot)\|_\Lto^2 \\
&\qquad \quad +\alpha \log(2)(t+\tau)\nu 2^{-(s+\alpha t)\nu} \la v_\nu(t,\cdot) | \partial_x\big( [\Dn,T_a^m]\partial_x w(t,\cdot) \big) \ral \\
&\qquad \quad +\alpha \log(2)(t+\tau)\nu 2^{-(s+\alpha t)\nu} \la v_\nu(t,\cdot) | \Dn \partial_x\big( (a-T_a^m)\partial_x w(t,\cdot) \big) \ral \\
&\qquad \quad -(t+\tau) 2^{-(s+\alpha t)\nu} \la \partial_x([\Dn, T_a^m]\partial_x w(t,\cdot)) | \partial_t v_\nu(t,\cdot) \ral \\
&\qquad \quad -(t+\tau) 2^{-(s+\alpha t)\nu} \la \Delta_\nu \partial_x((a-T_a^m)\partial_x w(t,\cdot)) | \partial_t v_\nu(t,\cdot) \ral. \end{aligned} \end{equation*} Integration by parts with respect to $t$ yields \begin{equation*}
\gamma(t+\tau)\la v_\nu(t,\cdot) | \partial_t v_\nu(t,\cdot) \ral = \frac{\gamma}{2}\frac{d}{dt}\Big( (t+\tau)\|v_\nu(t,\cdot)\|_\Lto^2 \Big) - \frac{\gamma}{2}\|v_\nu(t,\cdot)\|_\Lto^2
\end{equation*} and  \begin{equation*} \begin{aligned}
&-(t+\tau)\Phi_\lambda'\left( \frac{t+\tau}{\beta} \right) \la v_\nu(t,\cdot) | \partial_t v_\nu(t,\cdot) \ral = \\
&\qquad -\frac{1}{2}\frac{d}{dt}\Big( (t+\tau)\Phi_\lambda'\left( \frac{t+\tau}{\beta} \right) \|v_\nu(t,\cdot)\|_\Lto^2 \Big) + \frac{1}{2}\frac{t+\tau}{\beta}\Phi_\lambda''\left( \frac{t+\tau}{\beta} \right)\|v_\nu(t,\cdot)\|^2_\Lto \\
&\qquad + \frac{1}{2}\Phi_\lambda'\left( \frac{t+\tau}{\beta} \right)\|v_\nu(t,\cdot)\|^2_{L^2(\R^n)}.
\end{aligned}\end{equation*} Next we investigate the term $-(t+\tau)\la \partial_x(T_{a_\nu}^m\partial_x v_\nu(t,\cdot)) | \partial_t v_\nu(t,\cdot)\ral$. From \eqref{PPDef} it can be seen that $\partial_t T_{a_\nu}^m = T_{\partial_t a_\nu}^m + T_a^m\partial_t$. A simple computation shows that\begin{equation*} \begin{aligned}
&-(t+\tau)\la \partial_x(T_{a_\nu}^m\partial_x v_\nu(t,\cdot)) | \partial_t v_\nu(t,\cdot)\ral = \\
& \qquad  \frac{1}{2}\frac{d}{dt}\Big( (t+\tau)\la T_{a_\nu}^m \partial_x v_\nu(t,\cdot) | \partial_x v_\nu(t,\cdot) \ral \Big) \\
& \qquad - \frac{1}{2}\la  T_{a_\nu}^m \partial_x v_\nu(t,\cdot) | \partial_x v_\nu(t,\cdot) \ral \\
& \qquad -\frac{1}{2}(t+\tau)\la T_{\partial_t a_\nu}^m \partial_x v_\nu(t,\cdot) | \partial_x v_\nu(t,\cdot) \ral \\
& \qquad -\frac{1}{2}(t+\tau)\la \partial_t \partial_x v_\nu(t,\cdot) | ((T_{a_\nu}^m)^\ast-T_{a_\nu}^m)\partial_x v_\nu(t,\cdot) \ral.
\end{aligned} \end{equation*} Therefore we have: \begin{equation} \label{eq:EnEstimate} \begin{aligned}
&(t+\tau)\|\partial_t v_\nu(t,\cdot)\|_{L^2}^2 = \\
& \qquad \frac{\gamma}{2}\frac{d}{dt}\left( (t+\tau)\|v_\nu(t,\cdot)\|_\Lto^2 \right) - \frac{\gamma}{2}\|v_\nu(t,\cdot)\|_\Lto^2 \\
& \qquad -\frac{1}{2}\frac{d}{dt}\left( (t+\tau)\Phi_\lambda'\left( \frac{t+\tau}{\beta} \right)\|v_\nu(t,\cdot)\|_\Lto^2 \right) \\
& \qquad +\frac{1}{2}\Phi_\lambda'\left( \frac{t+\tau}{\beta} \right)\|v_\nu(t,\cdot)\|_\Lto^2 + \frac{1}{2}\frac{t+\tau}{\beta}\Phi_\lambda''\left( \frac{t+\tau}{\beta} \right)\|v_\nu(t,\cdot)\|_\Lto^2 \\
& \qquad - (t+\tau)\la \partial_x((T_a^m-T_{a_\nu}^m)\partial_x v_\nu(t,\cdot)) | \partial_t v_\nu(t,\cdot) \ral \\
& \qquad + \frac{1}{2}\frac{d}{dt}\left( (t+\tau)\la T_{a_\nu}^m \partial_x v_\nu(t,\cdot) | \partial_x v_\nu(t,\cdot) \ral \right) \\
& \qquad - \frac{1}{2}\la T_{a_\nu}^m \partial_x v_\nu(t,\cdot) | \partial_x v_\nu(t,\cdot) \ral \\
& \qquad - \frac{1}{2}(t+\tau)\la \partial_x v_\nu(t,\cdot) | T_{\partial_t a_\nu}^m \partial_x v_\nu(t,\cdot) \ral \\
& \qquad - \frac{1}{2}(t+\tau)\la \partial_t \partial_x v_\nu(t,\cdot) | ((T_{a_\nu}^m)^\ast - T_{a_\nu}^m)\partial_x v_\nu(t,\cdot) \ral \\
& \qquad - \alpha\gamma \log(2) (t+\tau) \nu \|v_\nu(t,\cdot)\|_\Lto^2 \\
& \qquad + \alpha\log(2)(t+\tau)\Phi_\lambda'\left( \frac{t+\tau}{\beta} \right)\nu \|v_\nu(t,\cdot)\|^2_\Lto \\
& \qquad - \alpha\log(2)(t+\tau)\nu \la \partial_x v_\nu(t,\cdot) | T_a^m \partial_x v_\nu(t,\cdot) \ral \\
& \qquad + \alpha^2(\log(2))^2 (t+\tau)\nu^2 \|v_\nu(t,\cdot)\|_{L^2} \\
& \qquad + \alpha\log(2)\nu 2^{-(s+\alpha t)\nu}(t+\tau)\la v_\nu(t,\cdot) | \pazocal X_\nu(t,\cdot) \ral \\
& \qquad - (t+\tau)2^{-(s+\alpha t)\nu}\la \pazocal X_\nu(t,\cdot) | \partial_t v_\nu(t,\cdot) \ral,
\end{aligned}\end{equation} where we have set \begin{eqnarray*}
\pazocal X_\nu(t,\cdot) &:=&  \left( \partial_x([\Dn,T_a^m]\partial_x w(t,\cdot)) + \Dn(\partial_x((a-T_a^m)\partial_x w(t,\cdot))) \right).
\end{eqnarray*}


\subsection{Estimates for $\nu = 0$}

Setting $\nu = 0$, we get from \eqref{eq:EnEstimate} \begin{equation*} \begin{aligned}
&(t+\tau)\|\partial_t v_0(t,\cdot)\|_{L^2}^2 = \\
& \qquad \frac{\gamma}{2}\frac{d}{dt}\left( (t+\tau)\|v_0(t,\cdot)\|_\Lto^2 \right) - \frac{\gamma}{2}\|v_0(t,\cdot)\|_\Lto^2 \\
& \qquad -\frac{1}{2}\frac{d}{dt}\left( (t+\tau)\Phi_\lambda'\left( \frac{t+\tau}{\beta} \right)\|v_0(t,\cdot)\|_\Lto^2 \right) \\
& \qquad +\frac{1}{2}\Phi_\lambda'\left( \frac{t+\tau}{\beta} \right)\|v_0(t,\cdot)\|_\Lto^2 + \frac{1}{2}\frac{t+\tau}{\beta}\Phi_\lambda''\left( \frac{t+\tau}{\beta} \right)\|v_0(t,\cdot)\|_\Lto^2 \\
& \qquad - (t+\tau)\la \partial_x((T_a^m-T_{a_0}^m)\partial_x v_0(t,\cdot)) | \partial_t v_0(t,\cdot) \ral \\
& \qquad + \frac{1}{2}\frac{d}{dt}\left( (t+\tau)\la \partial_x v_0(t,\cdot) | T_{a_0}^m \partial_x v_0(t,\cdot) \ral \right) \\
& \qquad - \frac{1}{2}\la \partial_x v_0(t,\cdot) | T_{a_0}^m \partial_x v_0(t,\cdot) \ral - \frac{1}{2}(t+\tau)\la \partial_x v_0(t,\cdot) | T_{\partial_t a_0}^m \partial_x v_0(t,\cdot) \ral \\
& \qquad - \frac{1}{2}(t+\tau)\la ((T_{a_0}^m)^\ast - T_{a_0}^m)\partial_x v_0(t,\cdot) | \partial_t \partial_x v_0(t,\cdot) \ral \\
& \qquad - (t+\tau) \la \pazocal X_0(t,\cdot)  | \partial_t v_0(t,\cdot) \ral. \\
\end{aligned}\end{equation*} Using Propositions \ref{PropBernstein}, \ref{MappingProp} and Lemma \ref{lem:Approx}, for $N_1$, $N_2>0$ we get \begin{eqnarray*}
&& |\la \partial_x v_0(t,\cdot) | T^m_{\partial_t a_0} \partial_x v_0(t,\cdot) \ral| \leq C_{a,m}^{(1)} \|v_0\|^2_\Lt, \\
&& |\la T_{a-a_0}^m\partial_x v_0(t,\cdot) | \partial_x\partial_t v_0(t,\cdot) \ral| \leq C_{a,m}^{(2)} N_1 \|v_0(t,\cdot)\|_\Lt^2 + \frac{1}{N_1} \|\partial_t v_0(t,\cdot)\|_\Lt^2, \\
&& |\la ((T_{a_0}^m)^\ast - T_{a_0}^m)\partial_x v_0(t,\cdot)| \partial_t \partial_x v_0(t,\cdot) \ral| \leq C_{a,m}^{(3)}N_2 \|v_0(t,\cdot)\|_\Lt^2 + \frac{1}{N_2} \|\partial_t v_0(t,\cdot)\|_\Lt^2.
\end{eqnarray*} Now we choose $N_1$ and $N_2$ so large that \begin{equation*}
\frac{1}{N_1} + \frac{1}{N_2} - \frac{1}{2} < 0
\end{equation*} and $\bar{\gamma}$ so large that \begin{equation*}
-\frac{\gamma}{4} + \left(C_{a,m}^{(1)} + C_{a,m}^{(2)} N_1 + C_{a,m}^{(3)} N_2\right)\Big(\frac{7}{8}\sigma + \tau\Big) < 0
\end{equation*} for $\gamma \geq \bar{\gamma}$. Hence, the term \begin{eqnarray*}
&& C_{a,m}^{(1)} (t+\tau) \|v_0(t,\cdot)\|_{\Lt}^2 + C_{a,m}^{(2)} N_1(t+\tau)\|v_0(t,\cdot)\|_\Lt^2 \\
&& + C_{a,m}^{(3)}N_2(t+\tau)\|v_0(t,\cdot)\|_\Lt^2
\end{eqnarray*} is absorbed by $-\frac{\gamma}{4} \|v_\nu(t,\cdot)\|_\Lt^2$ and the term \begin{equation*}
\frac{1}{N_1} (t+\tau) \|\partial_t v_0(t,\cdot)\|_\Lt^2 + \frac{1}{N_2} (t+\tau) \|\partial_t v_0(t,\cdot)\|_\Lt^2
\end{equation*} is absorbed by $-\frac{1}{2}(t+\tau)\|\partial_t v_0(t,\cdot)\|_\Lt^2$. Hence, we get \begin{equation*} \begin{aligned} 
& frac12(t+\tau)\|\partial_t v_0(t,\cdot)\|_\Lt^2 \\
& \qquad \leq \frac{\gamma}{2} \frac{d}{dt}\left( (t+\tau)\|v_0(t,\cdot)\|_\Lt^2 \right) - \frac{\gamma}{4}\|v_0(t,\cdot)\|_\Lt^2
 + \frac{1}{2} \Phi_\lambda'\left( \frac{t+\tau}{\beta} \right) \|v_0(t,\cdot)\|_\Lt^2 \\
& \qquad \quad  -\frac{1}{2}\frac{d}{dt}\left( (t+\tau) \Phi'_\lambda\left( \frac{t+\tau}{\beta} \right)\|v_0(t,\cdot)\|_\Lt^2 \right) + \frac{1}{2} \frac{t+\tau}{\beta}\Phi_\lambda''\left( \frac{t+\tau}{\beta} \right) \|v_0(t,\cdot)\|_\Lt^2 \\
& \qquad \quad + \frac{1}{2} \frac{d}{dt} \left( (t+\tau)\la \partial_x v_0(t,\cdot) | T_{a_0}^m \partial_x v_0(t,\cdot) \ra_\Lt \right) - (t+\tau) \la \pazocal X_0  | \partial_t v_0(t,\cdot) \ral.
\end{aligned} \end{equation*} Further, we recall that $\Phi$ fulfills equation \eqref{eq:DefPhi}, i.e. \begin{equation*}
y\Phi''_{\lambda}(y)= -\lambda(\Phi_\lambda'(y))^2 \mu\big( \frac{1}{\Phi_\lambda'(y)} \big) =
-\lambda\Phi_\lambda'(y)\big( 1 + |\log\big(\frac{1}{\Phi_\lambda'(y)}\big)| \big)
\end{equation*} for $\lambda > 1$. From this, we see that \begin{equation*}
\frac{1}{2} \Phi_\lambda'\left( \frac{t+\tau}{\beta} \right) \|v_0(t,\cdot)\|_\Lt^2 + \frac{1}{2} \frac{t+\tau}{\beta}\Phi_\lambda''\left( \frac{t+\tau}{\beta} \right) \|v_0(t,\cdot)\|_\Lt^2 < 0
\end{equation*} holds, and thus, we get \begin{equation*} \begin{aligned}
&\frac{\gamma}{8}\|v_0(t,\cdot)\|_\Lt^2 \\
& \qquad \quad \leq -\frac{1}{2}(t+\tau)\|\partial_t v_0(t,\cdot)\|_\Lt^2 + \frac{\gamma}{2} \frac{d}{dt}\left( (t+\tau)\|v_0(t,\cdot)\|_\Lt^2 \right) - \frac{\gamma}{8}\|v_0(t,\cdot)\|_\Lt^2 \\
& \qquad \qquad + \frac{1}{2} \frac{d}{dt} \left( (t+\tau)\la \partial_x v_0(t,\cdot) | T_{a_0}^m \partial_x v_0 \ra_\Lt \right) - (t+\tau) \la \pazocal X_0  | \partial_t v_0(t,\cdot) \ral \\
& \qquad \qquad -\frac{1}{2}\frac{d}{dt}\left( (t+\tau) \Phi'_\lambda\left( \frac{t+\tau}{\beta} \right)\|v_0(t,\cdot)\|_\Lt^2 \right).
\end{aligned}\end{equation*} Using Propositions \ref{prop:Pos} and \ref{MappingProp} as well as integrating in $t$ over $[0,p] \subseteq [0,\frac{7}{8}\sigma]$, we obtain \begin{equation*} \begin{aligned}
&\frac{\gamma}{8}\int_0^p \|v_0(t,\cdot)\|_\Lt^2 dt \\
& \qquad \leq (\frac{\gamma}{2} + C_{m,a}^{(4)})(p+\tau)\|v_0(p,\cdot)\|_\Lt^2 + \frac{1}{2} \tau \Phi_\lambda'\left( \frac{\tau}{\beta} \right) \|v_0(0,\cdot)\|_\Lt^2 \\
& \qquad \quad - \frac{\gamma}{8} \int_0^p \|v_0(t,\cdot)\|_\Lt^2 dt -\frac{1}{2} \int_0^p (t+\tau)\|\partial_t v_0(t,\cdot)\|_{\Lt}^2 dt \\
& \qquad \quad- \int_0^p (t+\tau)\la \pazocal X_0(t,\cdot) | \partial_t v_0(t,\cdot) \ra_\Lt dt,
\end{aligned}\end{equation*} where we have used \begin{eqnarray*}
|\la \partial_x v_0(p,\cdot)| T_{a_0}^m\partial_x v_0(p,\cdot) \ral| \leq C_{m,a}^{(4)}\|v_0(p,\cdot)\|_\Lt^2
\end{eqnarray*} and, applying Proposition \ref{prop:Pos}, \begin{equation*}
\la \partial_x v_0(t,\cdot) | T_{a_0}^m \partial_x v_0(t,\cdot) \ral \geq \frac{\kappa}{2}\|\partial_x v_0(t,\cdot)\|_\Lt^2
\end{equation*} choosing $m$ large enough.

\subsection{Estimates for $\nu \geq 1$}

Now we consider \eqref{eq:EnEstimate} for $\nu \geq 1$. From Lemma \ref{lem:Approx}, for $N_3$ and $N_4>0$ we obtain \begin{equation} \label{auxest4} \begin{aligned}
&|\la (T_a^m-T_{a_\nu}^m)\partial_x v_\nu(t,\cdot)| \partial_x\partial_t v_\nu(t,\cdot) \ral| \\
& \qquad \qquad \leq C_{a,m}^{(5)} N_3 \nu 2^{2\nu} \|v_\nu(t,\cdot)\|_\Lt^2 + \frac{1}{N_3}\|\partial_t v_\nu(t,\cdot)\|_\Lt^2,
\end{aligned} \end{equation} where we also used the fact that $(\mu(\eps))^2 \leq \mu(\eps)$, $\eps \in (0,1]$, and \begin{equation} \label{auxest3} \begin{aligned}
|\la \partial_x v_\nu(t,\cdot) | T_{\partial_t a_\nu}^m \partial_x v_\nu(t,\cdot)\ral| \leq C^{(6)}_{a,m}2^{2\nu}\|v_\nu(t,\cdot)\|^2_\Lt
\end{aligned} \end{equation} as well as \begin{equation} \label{auxest2} \begin{aligned}
& |\la ((T_{a_\nu}^m)^\ast-T_{a_\nu}^m)\partial_x v_\nu(t,\cdot) | \partial_t \partial_x v_\nu(t,\cdot) \ral| \\
& \qquad \qquad \leq C^{(7)}_{a,m}N_4 \|v_\nu(t,\cdot)\|_\Lt^2 + \frac{1}{N_4}\|\partial_t v_\nu(t,\cdot)\|_\Lt^2
\end{aligned} \end{equation} which follows from Proposition \ref{Adjoint}. Using again the positivity estimate in Proposition \ref{prop:Pos}, we obtain \begin{equation} \label{auxest1} \begin{aligned}
& - \alpha\log(2)(t+\tau)\nu \la \partial_x v_\nu(t,\cdot) | T_{a_\nu}^m \partial_x v_\nu(t,\cdot) \ral \\
& \qquad \qquad \qquad \qquad \leq -\alpha C^{(8)}_{a,m}(t+\tau)\nu 2^{2\nu} \|\partial_x v_\nu(t,\cdot)\|^2_{\Lt}.
\end{aligned} \end{equation} 
Now we choose $N_3$ and $N_4$ so large that \begin{equation*}
\frac{1}{N_3} + \frac{1}{N_4} - \frac{1}{2} < 0,
\end{equation*}
and $\alpha_1$ large enough so that \begin{equation*}
- \frac{\alpha_1}{2}C_{a,m}^{(8)} + N_3C_{a,m}^{(5)} + C_{a,m}^{(6)} + C_{a,m}^{(7)}N_4 < 0,
\end{equation*} and we set $\alpha := \max\{T^{-1}, \alpha_1\}$. 
 With this choice, we get \begin{equation} \label{finest1} \begin{aligned}
& \frac{\gamma}{4}\|v_\nu(t,\cdot)\|_\Lto^2 + \frac{1}{2} (t+\tau)\|\partial_t v_\nu(t,\cdot)\|_{L^2}^2 \\
& \qquad \leq \frac{\gamma}{2}\frac{d}{dt}\left( (t+\tau)\|v_\nu(t,\cdot)\|_\Lto^2 \right) - \frac{\gamma}{4}\|v_\nu(t,\cdot)\|_\Lto^2 \\
& \qquad\quad -\frac{1}{2}\frac{d}{dt}\left( (t+\tau)\Phi_\lambda'\left( \frac{t+\tau}{\beta} \right)\|v_\nu(t,\cdot)\|_\Lto^2 \right) \\
& \qquad\quad +\frac{1}{2}\Phi_\lambda'\left( \frac{t+\tau}{\beta} \right)\|v_\nu(t,\cdot)\|_\Lto^2 + \frac{1}{2}\frac{t+\tau}{\beta}\Phi_\lambda''\left( \frac{t+\tau}{\beta} \right)\|v_\nu(t,\cdot)\|_\Lto^2 \\
& \qquad\quad + \frac{1}{2}\frac{d}{dt}\left( (t+\tau)\la \partial_x v_\nu(t,\cdot) | T_{a_\nu}^m \partial_x v_\nu(t,\cdot) \ral \right) \\
& \qquad\quad - \alpha\gamma \log(2) (t+\tau) \nu \|v_\nu(t,\cdot)\|_\Lto^2 -\frac{1}{2}\la \partial_x v_\nu(t,\cdot) | T_{a_\nu}^m \partial_x v_\nu(t,\cdot) \ral \\
& \qquad\quad + \alpha\log(2)(t+\tau)\Phi_\lambda'\left( \frac{t+\tau}{\beta} \right)\nu \|v_\nu(t,\cdot)\|^2_\Lto \\
& \qquad\quad + \alpha^2(\log(2))^2 \nu^2 \|v_\nu(t,\cdot)\|_{L^2} -\frac{\alpha}{2} C^{(8)}_{a,m} (t+\tau) \nu 2^{2\nu} \|v_\nu(t,\cdot)\|_\Lt^2 \\
& \qquad\quad + \alpha\log(2)\nu 2^{-(s+\alpha t)\nu}(t+\tau)\la v_\nu(t,\cdot) | \pazocal X_\nu(t,\cdot) \ral \\
& \qquad\quad - (t+\tau)2^{-(s+\alpha t)\nu} \la \pazocal X_\nu(t,\cdot) | \partial_t v_\nu(t,\cdot) \ral.
\end{aligned}\end{equation} Since $y\Phi_\lambda''(y)=-\lambda \Phi_\lambda'(y)(1+|\log(\Phi_\lambda'(y))|)$, if we take $\lambda\geq\bar\lambda > 2$ we have \begin{equation*}
\frac{1}{4} \frac{t+\tau}{\beta} \Phi''_\lambda\left( \frac{t+\tau}{\beta} \right) \leq - \frac{1}{2}\Phi_\lambda'\left( \frac{t+\tau}{\beta} \right),
\end{equation*} and hence the term $\frac{1}{2}\Phi_\lambda'\left( \frac{t+\tau}{\beta} \right)\|v_\nu(t,\cdot)\|_\Lto^2$ in \eqref{finest1} is absorbed by the term $\frac{1}{4}\frac{t+\tau}{\beta}\Phi_\lambda''\left( \frac{t+\tau}{\beta} \right)\|v_\nu(t,\cdot)\|_\Lto^2$. Now we need to absorb \begin{equation} \label{term1}
\alpha\log(2)(t+\tau)\Phi_\lambda'\left( \frac{t+\tau}{\beta} \right)\nu \|v_\nu(t,\cdot)\|^2_\Lto.
\end{equation} There are two terms in \eqref{finest1} that will help to achieve this. One is \begin{equation} \label{term2}
-\frac{\alpha}{4} C^{(8)}_{a,m} (t+\tau) \nu 2^{2\nu} \|v_\nu(t,\cdot)\|_\Lt^2
\end{equation} and the other one is \begin{equation} \label{term3}
\frac{1}{4}\frac{t+\tau}{\beta}\Phi_\lambda''\left( \frac{t+\tau}{\beta} \right)\|v_\nu(t,\cdot)\|_\Lto^2.
\end{equation} Let $\tilde{C}_{a,m}^{(8)} = \min\{4\log(2), C_{a,m}^{(8)}\}$. If $\nu \geq (\log(2))^{-1}\log\big(\frac{4\log(2)}{\tilde{C}_{a,m}^{(8)}}\big)\Phi_\lambda'\left( \frac{t+\tau}{\beta} \right)$, then \begin{equation*}
-\frac{C_{a,m}^{(8)}}{4} \alpha \nu 2^{2\nu} \leq -\alpha\log(2)\Phi_\lambda'\left( \frac{t+\tau}{\beta} \right).
\end{equation*} If on the contrary $\nu < (\log(2))^{-1}\log\big(\frac{4\log(2)}{\tilde{C}_{a,m}^{(8)}}\big)\Phi_\lambda'\left( \frac{t+\tau}{\beta} \right)$ then $\Phi_\lambda'\left( \frac{t+\tau}{\beta} \right) > 2^{\nu}$ and, hence, by \eqref{eq:DefPhi}, we obtain \begin{equation*} \begin{aligned}
&\frac{1}{4}\frac{t+\tau}{\beta}\Phi_\lambda''\left(\frac{t+\tau}{\beta}\right) = -\frac{1}{4}\lambda \left( \Phi_\lambda'\left(\frac{t+\tau}{\beta}\right) \right)^2 \mu\left( \frac{1}{\Phi_\lambda'\left(\frac{t+\tau}{\beta}\right)} \right) \\
&\qquad \leq -\frac{1}{4}\lambda \left( \Phi_\lambda'\left(\frac{t+\tau}{\beta}\right) \right)^2 \mu\left( \frac{1}{\frac{4\log(2)}{\tilde{C}_{a,m}^{(8)}}\Phi_\lambda'\left(\frac{t+\tau}{\beta}\right)} \right) \\
&\qquad\leq -\frac{1}{4}\lambda \frac{\tilde{C}_{a,m}^{(8)}}{4\log(2)} \Phi_\lambda'\left(\frac{t+\tau}{\beta}\right)\left( 1+ \left| \log\left( \frac{4\log(2)}{\tilde{C}_{a,m}^{(8)}}\Phi_\lambda'\left(\frac{t+\tau}{\beta}\right) \right) \right| \right) \\
&\qquad\leq -\frac{1}{4}\lambda \frac{\tilde{C}_{a,m}^{(8)}}{4\log(2)} \Phi_\lambda'\left(\frac{t+\tau}{\beta}\right) (1+\nu\log(2)) \\
&\qquad\leq -\lambda C_{a,m}^{(9)} \Phi_\lambda'\left(\frac{t+\tau}{\beta}\right) \nu,
\end{aligned} \end{equation*} where we have used the fact  that the function $\mu$ is increasing. Consequently, if we choose $\lambda\geq\bar\lambda$ with \begin{equation*}
\bar\lambda \geq \frac{\alpha \log(2)\big(\frac{7}{8}\sigma+\tau\big)}{C_{a,m}^{(9)}},
\end{equation*} we have \begin{equation*}
\frac{1}{4}\frac{t+\tau}{\beta}\Phi''_\lambda\left( \frac{t+\tau}{\beta} \right) \leq -\alpha\log(2)(t+\tau)\Phi'_\lambda\left( \frac{t+\tau}{\beta} \right)
\end{equation*} and hence, the term \eqref{term1} is compensated by \eqref{term2} and \eqref{term3}. Now we consider the term \begin{equation} \label{term5}
\alpha^2\log^2(2) \nu^2 \|v_\nu(t,\cdot)\|_{L^2}.
\end{equation} If $\nu \geq (\log(2))^{-1} \log\left( \frac{\alpha(2\log(2))^2}{C_{a,m}^{(8)}} \right) =: \bar{\nu}_1$ then \begin{equation*}
-\frac{C_{a,m}^{(8)}}{4}\alpha \nu 2^{2\nu} + \alpha^2 \log^2(2) \nu^2 \leq 0.
\end{equation*} If $\nu \leq \bar{\nu}_1$, we choose $\bar{\gamma}$ possibly larger such that \begin{equation*}
\frac{\gamma}{4} \geq \alpha^2 \log^2(2) \bar{\nu}_1^2 \big( \frac{7}{8}\sigma + \tau \big)
\end{equation*} for all $\gamma \geq \bar{\gamma}$. We obtain \begin{equation*}
-\frac{\gamma}{4} + \alpha^2 \log^2(2) \nu \leq 0, 
\end{equation*} and consequently \eqref{term5} is absorbed by \begin{equation*}
-\frac{\alpha}{4} C^{(8)}_{a,m} (t+\tau) \nu 2^{2\nu} \|v_\nu(t,\cdot)\|_\Lt^2 - \frac{\gamma}{4}\|v_\nu(t,\cdot)\|_\Lto^2.
\end{equation*} The term $- \alpha\gamma \log(2) (t+\tau) \nu \|v_\nu(t,\cdot)\|_\Lto^2$ can be neglected since it is negative. However, we stress here that it is a crucial term in order to achieve our energy estimate for an equation including also lower order terms. Eventually,  recalling also Propositions \ref{PropBernstein} and \ref{prop:Pos}, we obtain \begin{equation*} \begin{aligned}
& \frac{1}{2}(t+\tau)\|\partial_t v_\nu(t,\cdot)\|_{L^2}^2 + \frac{\gamma}{8}\|v_\nu(t,\cdot)\|_\Lto^2 \\
& \qquad \leq \frac{\gamma}{2}\frac{d}{dt}\left( (t+\tau)\|v_\nu(t,\cdot)\|_\Lto^2 \right) - \frac{1}{2}\frac{d}{dt}\left( (t+\tau)\Phi_\lambda'\left( \frac{t+\tau}{\beta} \right)\|v_\nu(t,\cdot)\|_\Lto^2 \right) \\
& \qquad \quad  + \frac{1}{2}\frac{d}{dt}\left( (t+\tau)\la \partial_x v_\nu(t,\cdot) | T_{a_\nu}^m \partial_x v_\nu(t,\cdot) \ral \right) \\
& \qquad \quad - \frac\kappa82^{2\nu}\|v_\nu(t,\cdot)\|_{L^2}^2\\
& \qquad \quad - \frac{\alpha}{2}\log(2)C_{a,m}^{(8)}(t+\tau)\nu 2^{2\nu}\|v_\nu(t,\cdot)\|_\Lt^2 \\
& \qquad \quad + \alpha\log(2)\nu 2^{-(s+\alpha t)\nu}(t+\tau)\la v_\nu(t,\cdot) | \pazocal X_\nu(t,\cdot) \ral \\
& \qquad \quad - (t+\tau)2^{-(s+\alpha t)\nu} \la \pazocal X_\nu(t,\cdot) | \partial_t v_\nu(t,\cdot) \ral
- \frac{\gamma}{8}\|v_\nu(t,\cdot)\|_\Lto^2.
\end{aligned}\end{equation*}  Integrating over $[0,p] \subseteq [0,\frac{7}{8}\sigma]$, we get \begin{equation*} \begin{aligned}
& \frac{\kappa}{8} \int_0^p 2^{2\nu}\|v_\nu(t,\cdot)\|_\Lt^2 dt + \frac{\gamma}{8} \int_0^p \|v_\nu(t,\cdot)\|_\Lt^2 dt \\
& \qquad \leq \tau\Phi_\lambda'\left( \frac{\tau}{\beta} \right)\|v_\nu(0,\cdot)\|_\Lto^2 + \left(\frac{\gamma}{2}+C_{a,m}^{(10)} 2^{2\nu}\right)(p+\tau)\|v_\nu(p,\cdot)\|_\Lt^2 \\
& \qquad \quad - \frac{\alpha}{2}\log(2)C_{a,m}^{(8)} \int_0^p (t+\tau)\nu 2^{2\nu}\|v_\nu(t,\cdot)\|_\Lt^2 dt - \frac{\gamma}{8} \int_0^p \|v_\nu(t,\cdot)\|_\Lto^2 dt \\
& \qquad \quad  - \frac{1}{2}\int_0^p (t+\tau)\|\partial_t v_\nu(t,\cdot)\|_{L^2}^2 dt \\
& \qquad \quad - \int_0^p (t+\tau)2^{-(s+\alpha t)\nu} \la \pazocal X_\nu(t,\cdot) | \partial_t v_\nu(t,\cdot) \ral dt \\
& \qquad \quad + \alpha\log(2) \int_0^p \nu 2^{-(s+\alpha t)\nu}(t+\tau)\la v_\nu(t,\cdot) | \pazocal X_\nu(t,\cdot) \ral dt.
\end{aligned}\end{equation*} Now we sum over $\nu$ and we obtain \begin{equation*} \begin{aligned}
& \frac{\kappa}{8} \int_0^p \sum\limits_{\nu \geq 0}2^{2\nu}\|v_\nu(t,\cdot)\|_\Lt^2 dt + \frac{\gamma}{8} \int_0^p \sum\limits_{\nu \geq 0} \|v_\nu(t,\cdot)\|_\Lt^2 dt \\
& \qquad \leq \tau\Phi_\lambda'\left( \frac{\tau}{\beta} \right)\sum\limits_{\nu \geq 0}\|v_\nu(0,\cdot)\|_\Lto^2 - \frac{\gamma}{8} \int_0^p \sum\limits_{\nu \geq 0} \|v_\nu(t,\cdot)\|_\Lto^2 dt \\
& \qquad \quad - \frac{1}{2}\int_0^p (t+\tau) \sum\limits_{\nu \geq 0} \|\partial_t v_\nu(t,\cdot)\|_{L^2}^2 dt \\
& \qquad \quad + \frac{\gamma}{2}(p+\tau)\sum\limits_{\nu \geq 0}\|v_\nu(p,\cdot)\|_\Lt^2 + C_{a,m}^{(10)} (p+\tau)\sum\limits_{\nu \geq 0} 2^{2\nu}\|v_\nu(p,\cdot)\|_\Lt^2 \\
& \qquad \quad - \frac{\alpha}{2}\log(2)C_{a,m}^{(8)} \int_0^p (t+\tau)\sum\limits_{\nu \geq 0}\nu 2^{2\nu}\|v_\nu(t,\cdot)\|_\Lt^2 dt \\
& \qquad \quad - \int_0^p (t+\tau) \sum\limits_{\nu \geq 0} 2^{-(s+\alpha t)\nu} \la \pazocal X_\nu(t,\cdot) | \partial_t v_\nu(t,\cdot) \ral dt \\
& \qquad \quad + \alpha\log(2) \int_0^p (t+\tau) \sum\limits_{\nu \geq 0} \nu 2^{-(s+\alpha t)\nu}\la v_\nu(t,\cdot) | \pazocal X_\nu(t,\cdot) \ral dt.
\end{aligned}\end{equation*} 
Using the results from Sections \ref{sec:Aux1} and \ref{sec:Aux2}, we have the estimates \begin{equation*} \begin{aligned}
& - \int_0^p (t+\tau) \sum\limits_{\nu \geq 0} 2^{-(s+\alpha t)\nu} \la \pazocal X_\nu(t,\cdot) | \partial_t v_\nu(t,\cdot) \ral dt \\
& \qquad \leq \eta \int_0^p (t+\tau) \sum\limits_{\nu \geq 0} \|\partial_t v_\nu(t,\cdot)\|_\Lt^2 dt \\
& \qquad \quad + \frac{C_{a,m,s}^{(11)}}{\eta} \int_0^p (t+\tau) \sum\limits_{\nu \geq 0} 2^{2\nu} \|v_\nu(t,\cdot)\|_\Lt^2 dt
\end{aligned}\end{equation*} and \begin{equation*} \begin{aligned}
&\alpha\log(2) \int_0^p (t+\tau) \sum\limits_{\nu \geq 0} \nu 2^{-(s+\alpha t)\nu}\la v_\nu(t,\cdot) | \pazocal X_\nu(t,\cdot) \ral dt \\
& \qquad \leq \alpha\log(2) C^{12}_{a,m,s}\int_0^p (t+\tau) \sum\limits_{\nu \geq 0} 2^{2\nu}\|v_\nu(t,\cdot)\|_\Lt^2 dt.
\end{aligned}\end{equation*}

\subsection{End of the proof}

So far we have obtained \begin{equation*} \begin{aligned}
& \frac{\kappa}{8} \int_0^p \sum\limits_{\nu \geq 0}2^{2\nu}\|v_\nu(t,\cdot)\|_\Lt^2 dt + \frac{\gamma}{8} \int_0^p \sum\limits_{\nu \geq 0} \|v_\nu(t,\cdot)\|_\Lt^2 dt \\
& \qquad \leq \tau\Phi_\lambda'\left( \frac{\tau}{\beta} \right)\sum\limits_{\nu \geq 0}\|v_\nu(0,\cdot)\|_\Lto^2 - \frac{\gamma}{8} \int_0^p \sum\limits_{\nu \geq 0} \|v_\nu(t,\cdot)\|_\Lto^2 dt \\
& \qquad \quad + \frac{\gamma}{2}(p+\tau)\sum\limits_{\nu \geq 0}\|v_\nu(p,\cdot)\|_\Lt^2 + C_{a,m}^{(10)} (p+\tau)\sum\limits_{\nu \geq 0} 2^{2\nu}\|v_\nu(p,\cdot)\|_\Lt^2 \\
& \qquad \quad - \frac{\alpha}{2}\log(2)C_{a,m}^{(8)} \int_0^p (t+\tau)\sum\limits_{\nu \geq 0}\nu 2^{2\nu}\|v_\nu(t,\cdot)\|_\Lt^2 dt \\
& \qquad \quad  - \frac{1}{2}\int_0^p (t+\tau) \sum\limits_{\nu \geq 0} \|\partial_t v_\nu(t,\cdot)\|_{L^2}^2 dt + \eta \int_0^p (t+\tau) \sum\limits_{\nu \geq 0} \|\partial_t v_\nu(t,\cdot)\|_\Lt^2 dt \\
& \qquad \quad + \left(\alpha\log(2) C^{12}_{a,m,s}+\frac{C_{a,m,s}^{(11)}}{\eta}\right) \int_0^p (t+\tau) \sum\limits_{\nu \geq 0} 2^{2\nu}\|v_\nu(t,\cdot)\|_\Lt^2 dt.
\end{aligned}\end{equation*} Now we take $\eta < \frac{1}{2}$ and choose $\bar{\nu}_2 := \left \lceil{\left(\alpha\log(2) C^{12}_{a,m,s}+\frac{C_{a,m,s}^{(11)}}{\eta}\right) \frac{2}{\alpha \log(2) C_{a,m}^{(8)}}}\right \rceil$. With this, we have \begin{eqnarray*}
&& - \frac{\alpha}{2}\log(2)C_{a,m}^{(8)} \int_0^p (t+\tau)\sum\limits_{\nu \geq \bar{\nu}_2}\nu 2^{2\nu}\|v_\nu(t,\cdot)\|_\Lt^2 dt \\
&& \qquad \left(\alpha\log(2) C^{12}_{a,m,s}+\frac{C_{a,m,s}^{(11)}}{\eta}\right) \int_0^p (t+\tau) \sum\limits_{\nu \geq \bar{\nu}_2} 2^{2\nu}\|v_\nu(t,\cdot)\|_\Lt^2 dt \leq 0.
\end{eqnarray*} To absorb the remaining parts of the sum, we choose $\bar{\gamma}$ larger (if nescessary) such that \begin{equation*}
-\frac{\gamma}{8} + \left( \frac{7}{8}\sigma + \tau \right)\left(\alpha\log(2) C^{12}_{a,m,s}+\frac{C_{a,m,s}^{(11)}}{\eta}\right) 2^{2\bar{\nu}_2} < 0
\end{equation*} for all $\gamma \geq \bar{\gamma}$. This leads to \begin{equation*} \begin{aligned}
& - \frac{\gamma}{8} \int_0^p \sum\limits_{\nu < \bar{\nu}_2} \|v_\nu(t,\cdot)\|_\Lto^2 dt \\
& \qquad + \left(\alpha\log(2) C^{12}_{a,m,s}+\frac{C_{a,m,s}^{(11)}}{\eta}\right) \int_0^p (t+\tau) \sum\limits_{\nu < \bar{\nu}_2} 2^{2\nu}\|v_\nu(t,\cdot)\|_\Lt^2 dt \leq 0.
\end{aligned}\end{equation*} All in all, we finally obtain \begin{equation*} \begin{aligned}
& \frac{\kappa}{8} \int_0^p \sum\limits_{\nu \geq 0} 2^{2\nu}\|v_\nu(t,\cdot)\|_\Lt^2 dt + \frac{\gamma}{8} \int_0^p \sum\limits_{\nu \geq 0} \|v_\nu(t,\cdot)\|_\Lt^2 dt \\
& \qquad \leq \tau\Phi_\lambda'\left( \frac{\tau}{\beta} \right)\sum\limits_{\nu \geq 0}\|v_\nu(0,\cdot)\|_\Lto^2 + \frac{\gamma}{2}(p+\tau)\sum\limits_{\nu \geq 0}\|v_\nu(p,\cdot)\|_\Lt^2 \\
& \qquad \quad + C_{a,m}^{(10)} (p+\tau)\sum\limits_{\nu \geq 0} 2^{2\nu}\|v_\nu(p,\cdot)\|_\Lt^2.
\end{aligned}\end{equation*} From this, going back to $u_\nu$, we have \begin{equation*} \begin{aligned}
&\frac{\kappa}{8} \int_0^p   e^{2\gamma t} e^{-2\beta\Phi_\lambda\left(\frac{t+\tau}{\beta}\right)} \sum\limits_{\nu \geq 0} 2^{2(1-s-\alpha t)\nu}\|u_\nu(t,\cdot)\|_\Lt^2 dt \\
& \qquad + \frac{\gamma}{8} \int_0^p e^{2\gamma t} e^{-2\beta\Phi_\lambda\left(\frac{t+\tau}{\beta}\right)} \sum\limits_{\nu \geq 0} 2^{-2(s+\alpha t)\nu}\|u_\nu(t,\cdot)\|^2_\Lt dt \\
& \leq C_{a,m}^{(10)} (p+\tau)e^{2\gamma p} e^{-2\beta\Phi_\lambda\left(\frac{p+\tau}{\beta}\right)} \sum\limits_{\nu \geq 0} 2^{2(1-s-\alpha p)\nu}\|u_\nu(p,\cdot)\|_\Lt^2 \\
& \qquad + \frac{\gamma}{2}(p+\tau)e^{2\gamma p} e^{-2\beta\Phi_\lambda\left(\frac{p+\tau}{\beta}\right)} \sum\limits_{\nu \geq 0} 2^{-(s+\alpha p)\nu} \|u_\nu(p,\cdot)\|_\Lt^2 \\
& \qquad + \tau\Phi_\lambda'\left( \frac{\tau}{\beta} \right) e^{-2\beta\Phi_\lambda\left(\frac{\tau}{\beta}\right)} \sum\limits_{\nu \geq 0} 2^{-2s\nu} \|u_\nu(0,\cdot)\|_\Lto^2.
\end{aligned}\end{equation*} Using Proposition \ref{SobolevLP}, the weighted energy estimate \eqref{MainIn} follows. \qed


\section{Proof of Theorem \ref{thm:Stab}} \label{sec:Derivation}

In this section we show how the stability estimate in Theorem \ref{thm:Stab} follows from the energy estimate in Theorem \ref{thm:Energy}.
To this end, we need two lemmas whose proof is left to the reader. 

\begin{lem} \label{lem:aux3}  There exists $\gamma_0>0$ such that if $\gamma \geq \gamma_0$ then, for every $u \in \pazocal H$ solution of \eqref{BWeq}, the function $E(t) = e^{2\gamma t}\|u(t,\cdot)\|_{L^2}^2$ is not decreasing in $[0,T]$. \end{lem}

The next lemma contains an estimate of the $H^1$-norm of a solution of \eqref{BWeq} by its $L^2$-norm. This estimate is crucial in gaining \eqref{est:Stab} from \eqref{EnergyEst1}.

\begin{lem} \label{lem:aux4} There axist a constant $C$ such that, if $u \in \pazocal H$ be a solution of \eqref{BWeq} in $[0,T]$, then
\begin{equation*}
\inf_{t\in[\frac57\sigma,\frac78\sigma]}\|u(t)\|_{H^1} ^2\leq \frac C\sigma\sup_{t\in[\frac57\sigma,\frac78\sigma]} \|u(t)\|_{L^2}^2.
\end{equation*} The constant $C$ depends only on $\kappa$.\end{lem}

We start from the inequality \begin{eqnarray*}
&& \int_0^p e^{2\gamma t} e^{-2\beta \Phi_\lambda\left( \frac{t+\tau}{\beta} \right)} \|u(t,\cdot)\|^2_{H^{1-s-\alpha t}} dt \\
&& \qquad \leq M\Big[ (p+\tau)e^{2\gamma p}e^{-2\beta \Phi_\lambda\left( \frac{p+\tau}{\beta} \right)} \|u(p,\cdot)\|^2_{H^{1-s-\alpha p}} \\
&& \qquad \qquad + \tau \Phi_\lambda'\left( \frac{\tau}{\beta} \right) e^{-2\beta \Phi_\lambda\left( \frac{\tau}{\beta} \right)} \|u(0,\cdot)\|^2_{H^{-s}} \Big]
\end{eqnarray*} which is valid for $p \in [0,\frac{7}{8}\sigma]$, $\sigma := \frac{1-s}{\alpha}$. For every $\sigma^\ast \in (\frac{5}{8}\sigma,\frac{7}{8}\sigma)$, we have \begin{eqnarray*}
&& \int_0^{\sigma^\ast} e^{2\gamma t} e^{-2\beta \Phi_\lambda\left( \frac{t+\tau}{\beta} \right)} \|u(t,\cdot)\|^2_{H^{1-s-\alpha t}} dt \\
&& \qquad \leq M\Big[ (\sigma^\ast+\tau)e^{2\gamma \sigma^\ast}e^{-2\beta \Phi_\lambda\left( \frac{\sigma^\ast+\tau}{\beta} \right)} \|u(\sigma^\ast,\cdot)\|^2_{H^{1-s-\alpha \sigma^\ast}} \\
&& \qquad \qquad + \tau \Phi_\lambda'\left( \frac{\tau}{\beta} \right) e^{-2\beta \Phi_\lambda\left( \frac{\tau}{\beta} \right)} \|u(0,\cdot)\|^2_{H^{-s}} \Big],
\end{eqnarray*} where $\beta \geq \sigma + \tau$. Now we take $p \in [0,\bar{\sigma}]$ with $\bar{\sigma} = \frac{1}{2}\left(\frac{\sigma}{2}-\tau \right)= \frac{\sigma}{8}$, so $2p+\tau \leq 2\bar{\sigma}+\tau = \frac{\sigma}{2} < \frac{5}{8}\sigma < \sigma^\ast$, and hence \begin{eqnarray*}
&& \int_p^{2p+\tau} e^{2\gamma t} e^{-2\beta \Phi_\lambda\left( \frac{t+\tau}{\beta} \right)} \|u(t,\cdot)\|^2_{H^{1-s-\alpha t}} dt \\
&& \qquad \leq M\Big[ (\sigma^\ast+\tau)e^{2\gamma \sigma^\ast}e^{-2\beta \Phi_\lambda\left( \frac{\sigma^\ast+\tau}{\beta} \right)} \|u(\sigma^\ast,\cdot)\|^2_{H^{1-s-\alpha \sigma^\ast}} \\
&& \qquad \qquad + \tau \Phi_\lambda'\left( \frac{\tau}{\beta} \right) e^{-2\beta \Phi_\lambda\left( \frac{\tau}{\beta} \right)} \|u(0,\cdot)\|^2_{H^{-s}} \Big].
\end{eqnarray*} Since $\frac{1}{8}(1-s) \leq 1-s-\alpha t \leq 1-s$, we have \begin{equation*}
\|u(t,\cdot)\|_{H^{1-s-\alpha t}} \geq \|u(t,\cdot)\|_{L^2}.
\end{equation*} Hence, with Lemma \ref{lem:aux3}, \begin{eqnarray*}
&& e^{2\gamma p} (p+\tau) \|u(p,\cdot)\|^2_{L^2} e^{-2\beta \Phi_\lambda\left( \frac{2p+2\tau}{\beta} \right)} \\
&& \qquad \leq M\Big[ (\sigma^\ast+\tau)e^{2\gamma \sigma^\ast}e^{-2\beta \Phi_\lambda\left( \frac{\sigma^\ast+\tau}{\beta} \right)} \|u(\sigma^\ast,\cdot)\|^2_{H^{1-s-\alpha \sigma^\ast}} \\
&& \qquad \qquad + \tau \Phi_\lambda'\left( \frac{\tau}{\beta} \right) e^{-2\beta \Phi_\lambda\left( \frac{\tau}{\beta} \right)} \|u(0,\cdot)\|^2_{H^{-s}} \Big].
\end{eqnarray*} 
Since $\Phi'_{\lambda}\geq1$, we have \begin{eqnarray*}
\|u(p,\cdot)\|_{L^2}^2 &\leq& M \frac{\sigma^\ast+\tau}{\tau} e^{2\gamma \sigma^\ast} \Phi_\lambda'\left( \frac{\tau}{\beta} \right) \\
&& \qquad \times \Big[
e^{2\beta\Phi_\lambda\left(\frac{\sigma/2+\tau}{\beta}\right)-2\beta\Phi_\lambda\left(\frac{\sigma^\ast+\tau}{\beta}\right)}\|u(\sigma^\ast,\cdot)\|^2_{H^{1-s-\alpha \sigma^\ast}} \\
&& \qquad \quad \qquad + e^{2\beta\Phi_\lambda\left(\frac{\sigma/2+\tau}{\beta}\right)-2\beta\Phi_\lambda\left(\frac{\tau}{\beta}\right)} \|u(0,\cdot)\|^2_{H^{-s}}
\Big] \\
 \qquad &\leq& \tilde{M} \Phi_\lambda'\left( \frac{\tau}{\beta} \right) e^{2\beta\Phi_\lambda\left( \frac{\sigma/2+\tau}{\beta} \right)-2\beta\Phi_\lambda'\left( \frac{\sigma^\ast+\tau}{\beta} \right)} \\
&& \qquad \qquad \times\left[ \|u(\sigma^\ast,\cdot)\|^2_{H^{1-s-\alpha \sigma^\ast}}
+ e^{-\beta\Phi_\lambda\left( \frac{\tau}{\beta} \right)}\|u(0,\cdot)\|^2_{H^{-s}}
\right].
\end{eqnarray*} Now we have $\frac{\sigma^\ast+\tau}{\beta} \geq \frac{\frac{5}{8}\sigma +\tau}{\beta}$, which implies \begin{equation*}
\Phi_\lambda\left( \frac{\sigma^\ast + \tau}{\beta} \right) \geq \Phi_\lambda\left( \frac{\frac{5}{8}\sigma + \tau}{\beta} \right)
\end{equation*} and hence \begin{eqnarray*}
\|u(p,\cdot)\|_{L^2}^2 &\leq& \tilde{M} \Phi_\lambda'\left( \frac{\tau}{\beta} \right) e^{2\beta\Phi_\lambda\left( \frac{\sigma/2+\tau}{\beta} \right)-2\beta\Phi_\lambda\left( \frac{5\sigma/8+\tau}{\beta} \right)} \\
&& \qquad \times \left[ \|u(\sigma^\ast,\cdot)\|^2_{H^{1-s-a\alpha \sigma^\ast}}
+ e^{-\beta\Phi_\lambda\left( \frac{\tau}{\beta} \right)}\|u(0,\cdot)\|^2_{H^{-s}}
\right].
\end{eqnarray*} By the concavity of $\Phi_\lambda$, we have \begin{eqnarray*}
&& 2\beta\Phi_\lambda\left( \frac{\sigma /2+\tau}{\beta} \right)-2\beta\Phi_\lambda\left( \frac{5\sigma /8+\tau}{\beta} \right) \\
&& \qquad \leq \Phi'_\lambda\left( \frac{5\sigma/8+\tau}{\beta} \right)\left( \frac{\sigma/2+\tau}{\beta} - \frac{5\sigma /8+\tau}{\beta} \right)
= -\Phi_\lambda'\left( \frac{5\sigma /8 +\tau}{\beta} \right)\frac{\sigma}{8 \beta}.
\end{eqnarray*} This implies \begin{eqnarray*}
\|u(p,\cdot)\|^2_{L^2} &\leq& \tilde{M} \Phi_\lambda'\left( \frac{\tau}{\beta} \right) e^{-\frac{\sigma}{4}\Phi_\lambda'\left( \frac{5\sigma/8+\tau}{\beta} \right)} \\
&& \qquad \times \left( \|u(\sigma^\ast,\cdot)\|^2_{H^{1-s-\alpha \sigma^\ast}} + e^{-2\beta\Phi_\lambda\left(\frac{\tau}{\beta}\right)}\|u(0,\cdot)\|_{H^{-s}} \right).
\end{eqnarray*}
By Lemma \ref{weight-properties}, we have \begin{eqnarray*}
\Phi_\lambda'\left( \frac{5\sigma /8 +\tau}{\beta} \right) &=& \Psi_\lambda\left( \frac{5\sigma/8 +\tau}{\tau}\frac{\tau}{\beta} \right) \\
&=& \exp\left( \left(\frac{5\sigma/8 +\tau}{\tau}\right)^{-\lambda}-1 \right)\Psi_\lambda\left( \frac{\tau}{\beta} \right)^{\left(\frac{5\sigma/8 +\tau}{\tau}\right)^{-\lambda}}.
\end{eqnarray*} 
Setting $\tilde{\delta} := \left(\frac{5\sigma/8 +\tau}{\tau}\right)^{-\lambda}$, we have \begin{equation*}
\|u(p,\cdot)\|^2_{L^2} \leq \tilde{M} \psi_\lambda\left( \frac{\tau}{\beta} \right)e^{-\tilde{N}\psi_\lambda\left( \frac{\tau}{\beta} \right)^{\tilde{\delta}}} \left( \|u(\sigma^\ast,\cdot)\|^2_{H^1} + e^{-2\beta\Phi_\lambda\left( \frac{\tau}{\beta} \right)}\|u(0,\cdot)\|^2_{H^{-s}} \right).
\end{equation*} 
Now we choose $\beta$ such that \begin{equation*}
e^{-\beta\Phi_\lambda\left( \frac{\tau}{\beta} \right)} = \|u(0,\cdot)\|^{-1}_{H^{-s}}
\end{equation*} 
that is
\begin{equation*}
\beta=\tau\Lambda^{-1}\left(\frac1\tau\log\|u(0,\cdot)\|_{L^2}\right).
\end{equation*}
Then there exists $\bar\rho>0$ such that, if $\|u(0,\cdot)\|_{L^2}\leq\bar\rho$, then $\beta\geq\sigma+\tau$. 
With this choice and thanks to Lemma \ref{weight-properties}, we get \begin{equation*}
\|u(p,\cdot)\|_{L^2}^2 \leq \tilde{\tilde{M}} \exp\left( -\tilde{\tilde{N}}\left[ \frac{1}{\tau}\log\left( \|u(0,\cdot)\|_{H^{-s}} \right) \right]^{\tilde\delta} \right) \left( \|u(\sigma^\ast,\cdot)\|^2_{H^1} + 1  \right)
\end{equation*} for all $\sigma^\ast \in [\frac{5}{8}\sigma,\frac{7}{8}\sigma]$ and for all $p \in [0, \frac{\sigma}{8}]$. By Lemma \ref{lem:aux4}, we finally get \begin{eqnarray*}
\|u(p,\cdot)\|_{L^2}^2 \leq C e^{-\tilde{\tilde{N}}\left[\frac{1}{\tau}\left| \log\left( \|u(0,\cdot)\|_{H^{-s}}\right)\right|\right]^{\tilde\delta}} \left( \max\limits_{t \in [\frac{5}{8}\sigma,\frac{7}{8}\sigma]} \|u(t,\cdot)\|_{L^2}^2 + 1 \right).
\end{eqnarray*} This completes the proof of Theorem \ref{thm:Stab}. \qed



\section*{Appendix}
\addcontentsline{toc}{section}{Appendix}

\begin{proof}[Proof of Proposition \ref{prop:AuxP1}]
To estimate \begin{equation*}
\sum\limits_{\nu \geq 0} 2^{-(s+\alpha t)\nu}\la \partial_{x_i} \partial_t v_\nu(t,\cdot) | \Delta_\nu((a-T_a^m)\partial_{x_j} w(t,\cdot)) \ral,
\end{equation*} we introduce a second microlocalization: Setting $w(t,\cdot) = \sum_{\mu \geq 0} w_\mu(t,\cdot)$ and $w_\mu(t,\cdot) = 2^{(s+\alpha t)\nu}v_\mu(t,\cdot)$ (see Section \ref{sec:Prel}) we obtain, using Proposition \ref{PropBernstein}, that \begin{eqnarray*}
&& \sum\limits_{\nu \geq 0} 2^{-(s+\alpha t)\nu}\la \partial_{x_i} \partial_t v_\nu(t,\cdot) | \Delta_\nu((a-T_a^m)\partial_{x_j} w(t,\cdot)) \ral \\
&& \qquad = \sum\limits_{\nu \geq 0}\sum\limits_{\mu \geq 0}  2^{-(s+\alpha t)\nu}\la \partial_{x_i} \partial_t v_\nu(t,\cdot)| \Delta_\nu((a-T_a^m)\partial_{x_j} w_\mu(t,\cdot)) \ral \\
&& \qquad = \sum\limits_{\nu \geq 0}\sum\limits_{\mu \geq 0}  2^{-(s+\alpha t)(\nu-\mu)}\la \partial_{x_i} \partial_t v_\nu(t,\cdot) | \Delta_\nu((a-T_a^m)\partial_{x_j} v_\mu(t,\cdot)) \ral \\
&& \qquad \leq \sum\limits_{\nu \geq 0}\sum\limits_{\mu \geq 0}  2^{-(s+\alpha t)(\nu-\mu)} 2^\nu \|\partial_t v_\nu(t,\cdot)\|_\Lt \|\Delta_\nu((a-T_a^m)\partial_{x_j} v_\mu(t,\cdot))\|_\Lt \\
&& \qquad \leq \sum\limits_{\nu \geq 0}\sum\limits_{\mu \geq 0}  2^{-(s+\alpha t)(\nu-\mu)} 2^\nu \|\partial_t v_\nu(t,\cdot)\|_\Lt \|\Delta_\nu\Omega_1\partial_{x_j} v_\mu(t,\cdot)\|_\Lt \\
&& \qquad \qquad + \sum\limits_{\nu \geq 0}\sum\limits_{\mu \geq 0}  2^{-(s+\alpha t)(\nu-\mu)} 2^\nu \|\partial_t v_\nu(t,\cdot)\|_\Lt \|\Delta_\nu\Omega_2\partial_{x_j} v_\mu(t,\cdot)\|_\Lt.
\end{eqnarray*} Since $w(t,\cdot) \in H^{1}(\R^n_x)$ we have $\partial_x v_\mu \in H^{-s}(\R^n_x)$ and, taking an $s' \in (0,s)$, also $\partial_x v_\mu(t,\cdot) \in H^{-s'}(\R_x^n)$. By Lemma \ref{lem:Omega1}, we then get \begin{equation*}
\|\Delta_\nu\Omega_1\partial_{x_j} v_\mu(t,\cdot)\|_\Lt \leq C c_\nu^{(\mu)} 2^{-(1-s')\nu} 2^\mu \|v_\mu(t,\cdot)\|_{L^2}
\end{equation*} and therefore \begin{eqnarray*}
&& \sum\limits_{\nu \geq 0}\sum\limits_{\mu \geq 0}  2^{-(s+\alpha t)(\nu-\mu)} 2^\nu \|\partial_t v_\nu(t,\cdot)\|_\Lt \|\Delta_\nu\Omega_1\partial_{x_j} v_\mu(t,\cdot)\|_\Lt \\
&& \qquad \leq C \sum\limits_{\nu \geq 0} \sum\limits_{\mu \leq \nu}  2^{-(s+\alpha t)(\nu-\mu)} 2^\nu \|\partial_t v_\nu(t,\cdot)\|_\Lt
c_\nu^{(\mu)} 2^{-(1-s')\nu} 2^\mu 2^{-s' \mu}\|v_\mu(t,\cdot)\|_\Lt \\
&& \qquad \leq C \sum\limits_{\nu \geq 0} \sum\limits_{\mu \leq \nu} 2^{-s\alpha t \nu}2^{s\alpha t \mu} \left( 2^{(s'-s)(\nu-\mu)}\|\partial_t v_\nu(t,\cdot)\|_\Lt \right) \left( c_\nu^{(\mu)} 2^\mu \|v_\mu(t,\cdot)\|_\Lt \right)\\
&& \qquad \leq \frac{1}{N} \sum\limits_{\nu \geq 0} \sum\limits_{\mu \leq \nu}  2^{-2(s-s')(\nu-\mu)}\|\partial_t v_\nu(t,\cdot)\|_\Lt^2 + C N \sum\limits_{\nu \geq 0} \sum\limits_{\mu \leq \nu}  (c_\nu^{(\mu)})^2 2^{2\mu} \|v_\mu(t,\cdot)\|_{L^2}^2 \\
&& \qquad \leq \frac{1}{N} \sum\limits_{\nu \geq 0} \Big( \sum\limits_{\mu \leq \nu} 2^{2(s-s')\mu} \Big)2^{-2(s-s')\nu} \|\partial_t v_\nu(t,\cdot)\|_\Lt^2 \\
&& \qquad \qquad + C N \sum\limits_{\mu \geq 0} \Big( \sum\limits_{\nu \geq 0}  (c_\nu^{(\mu)})^2 \Big) 2^{2\mu} \|v_\mu(t,\cdot)\|_{L^2}^2 \\
&& \qquad = \frac{1}{N}  \sum\limits_{\nu \geq 0} \frac{2^{2(s-s')(\nu+1)}-1}{2^{2(s-s')}-1} 2^{-2(s-s')\nu}\| \partial_t v_\nu(t,\cdot)\|_\Lt^2 + C N \sum\limits_{\mu \geq 0} 2^{2\mu}\| v_\mu(t,\cdot)\|_\Lt^2 \\
&& \qquad \leq \frac{1}{N}  \frac{2^{2(s-s')}}{2^{2(s-s')}-1} \sum\limits_{\nu \geq 0} \| \partial_t v_\nu(t,\cdot)\|_\Lt^2 + CN \sum\limits_{\mu \geq 0} 2^{2\mu}\|v_\mu(t,\cdot)\|_\Lt^2.
\end{eqnarray*} By the summation formula of the geometric sum and the integral criterion, we obtain \begin{equation*}
\frac{2^{2(s-s')}}{2^{2(s-s')}-1} \leq \frac{2^{2(s-s')}}{2^{2(s-s')}(1-2^{-2(s-s')})} \leq \frac{C}{s-s'}
\end{equation*} and, hence, \begin{eqnarray*}
&&\sum\limits_{\nu \geq 0}\sum\limits_{\mu \geq 0}  2^{-(s+\alpha t)(\nu-\mu)} 2^\nu \|\partial_t v_\nu(t,\cdot)\|_\Lt \|\Delta_\nu\Omega_1\partial_x v_\mu(t,\cdot)\|_\Lt \\
&& \qquad \qquad  \leq \frac{1}{N}  \frac{C}{s-s'} \sum\limits_{\nu \geq 0} \| \partial_t v_\nu(t,\cdot)\|_\Lt^2 + CN \sum\limits_{\mu \geq 0} 2^{2\mu}\|v_\mu(t,\cdot)\|_\Lt^2.
\end{eqnarray*} 
On the other hand, we have from Lemma \ref{lem:Omega2} that 
\begin{equation*}
\|\Delta_\nu\Omega_2\partial_x v_\mu(t,\cdot)\|_\Lt \leq C \tilde c_\nu^{(\mu)}  \|v_\mu(t,\cdot)\|_\Lt,
\end{equation*} 
and therefore we get 
\begin{eqnarray*}
&& \sum\limits_{\nu \geq 0}\sum\limits_{\mu \geq 0} 2^{-(s+\alpha t)(\nu-\mu)} 2^\nu \|\partial_t v_\nu(t,\cdot)\|_\Lt \|\Delta_\nu\Omega_2\partial_x v_\mu(t,\cdot)\|_\Lt \\
&& \qquad \leq C \sum\limits_{\nu \geq 0} \sum\limits_{\mu \geq \nu-4}  2^{-(s+\alpha t)(\nu-\mu)} 2^\nu \|\partial_t v_\nu(t,\cdot)\|_\Lt \|\Delta_\nu\Omega_2\partial_x v_\mu(t,\cdot)\|_\Lt \\
&& \qquad \leq C \sum\limits_{\nu \geq 0} \sum\limits_{\mu \geq \nu-4} 2^{-(s+\alpha t)(\nu-\mu)} 2^{\nu} \|\partial_t v_\nu(t,\cdot)\|_\Lt \tilde{c}_\nu^{(\mu)} 2^{-\mu}2^{\mu} \|v_\mu(t,\cdot)\|_\Lt \\
&& \qquad \leq C\sum\limits_{\nu \geq 0} \sum\limits_{\mu \geq \nu-4} 2^{(1-s-\alpha t)\nu} 2^{-(1-s-\alpha t)\mu} \tilde{c}_\nu^{(\mu)} \|\partial_t v_\nu(t,\cdot)\|_\Lt 2^\mu \|v_\mu(t,\cdot)\|_\Lt \\
&& \qquad \leq \frac{1}{N}\sum\limits_{\nu \geq 0} \sum\limits_{\mu \geq \nu-4} 2^{2(1-s-\alpha t)\nu} 2^{-2(1-s-\alpha t)\mu} \|\partial_t v_\nu(t,\cdot)\|_\Lt^2 \\
&& \qquad \qquad + CN \sum\limits_{\nu \geq 0} \sum\limits_{\mu \geq \nu-4} (c_\nu^{(\mu)})^2 2^{2\mu} \|v_\mu(t,\cdot)\|_\Lt^2 \\
&& \qquad \leq \frac{1}{N} \frac{2^{8(1-s-\alpha t)}}{1-2^{-2(1-s-\alpha t)}} \sum\limits_{\nu \geq 0} \|\partial_t v_\nu(t,\cdot)\|_\Lt^2 \\
&& \qquad \qquad + CN \sum\limits_{\mu \geq 0} \left( \sum\limits_{\nu \leq \mu +4} (c_\nu^{(\mu)})^2 \right) 2^{2\mu} \|v_\mu(t,\cdot)\|_\Lt^2.
\end{eqnarray*} Since $t \in [0,\frac{7}{8}\sigma]$, where $\sigma := \frac{1-s}{\alpha}$, we have $\frac{1}{8}(1-s) \leq 1-s-\alpha t \leq 1-s$ and hence \begin{equation*}
\frac{2^{8(1-s-\alpha t)}}{1-2^{-2(1-s-\alpha t)}} \leq \frac{C}{1-s-\alpha t} \leq \frac{C}{1-s}.
\end{equation*} From that, we get \begin{eqnarray*}
&& \sum\limits_{\nu \geq 0}\sum\limits_{\mu \geq 0} 2^{-(s+\alpha t)(\nu-\mu)} 2^\nu \|\partial_t v_\nu(t,\cdot)\|_\Lt \|\Delta_\nu\Omega_2\partial_x v_\mu(t,\cdot)\|_\Lt \\
&& \qquad \qquad \leq \frac{1}{N}\frac{C}{1-s} \sum\limits_{\nu \geq 0} \|\partial_t v_\nu(t,\cdot)\|_\Lt^2 + C N \sum\limits_{\mu \geq 0} 2^{2\mu}\|v_\mu(t,\cdot)\|_\Lt.
\end{eqnarray*} This concludes the proof of the proposition. \end{proof}

\begin{proof}[Proof of Lemma \ref{lemmacomm}] 
The proof is very similar to that of \cite{DSP2012}[Prop. 3.7]. We detail it for the reader's convenience. We have \begin{equation*}
[\Dn,T_a^m]w = [\Dn,S_{m-1}a]S_{m+1}w +  \sum\limits_{k \geq m+2} [\Dn, S_{k-3}a]\Delta_k w
\end{equation*} and get \begin{eqnarray*}
&& \partial_{x_j}[\Dn,T_a^m]\partial_{x_h}w = \partial_{x_j}([\Dn,S_{m-1}a]S_{m+1}(\partial_{x_h}w)) \\
&& \qquad \qquad \qquad \qquad \qquad  + \partial_{x_j}\Big( \sum\limits_{k \geq m+2} [\Dn,S_{k-3}a]\Delta_k(\partial_{x_h}w) \Big)
\end{eqnarray*} since $\Dn$ and $\Delta_k$ commute and we therefore have that \begin{eqnarray*}
&& [\Dn,S_{m-1}a]S_{m+1}w = \Dn(S_{m-1}aS_{m+1}w)-S_{m-1}aS_{m+1}(\Dn w) \\
&& \qquad \qquad = \Dn(S_{m-1}aS_{m+1}w) - S_{m-1}a\Dn(S_{m+1}w).
\end{eqnarray*} This holds analogously for $[\Dn, S_{k-3}a]\Delta_k w$. Let us consider \begin{equation*}
\partial_{x_j}([\Dn,S_{m-1}a]S_{m+1}(\partial_{x_h}w)) = \partial_{x_j}([\Dn,S_{m-1}a]\partial_{x_h}(S_{m+1}w)).
\end{equation*} Looking at the spectrum of this term  we see that the term equals to $0$ if $\nu \geq m+4$. Moreover, the spectrum is contained in $\{|\xi| \leq 2^{m+3}\}$. From Bernstein's inequality, we have that \begin{equation*}
\|\partial_{x_j}([\Dn,S_{m-1}a]S_{m+1}(\partial_{x_h}w))\|_\Lt \leq 2^{m+3} \|[\Dn,S_{m-1}a]S_{m+1}(\partial_{x_h}w)\|_\Lt.
\end{equation*} From the well known result of Coifman and Meyer  \cite[Th. 35]{CM1978}, which essentially say that
\begin{equation} \label{eq:CM}
\|[\Dn,{b}]\partial_{x}w\|_{\Lt} \leq C \|\nabla_x b\|_{L^\infty} \|w\|_{\Lt},
\end{equation} where $b \in \Lip(\R^n_x)$ and $w \in H^{1}(\R^n_x)$, we get \begin{equation*}
\|[\Dn,S_{m-1}a]\partial_{x_h}(S_{m+1}w)\|_\Lt \leq C\|a\|_{\Lip}\|S_{m+1}w\|_\Lt.
\end{equation*} Further, we have \begin{equation*}
\|S_{m+1}w\|_\Lt \leq \sum\limits_{k \leq m+1} \|\Delta_k w\|_{L^2} \leq C \sum\limits_{k \leq m+1} 2^{-(1-s-\alpha t)} \eps_k,
\end{equation*} where $\{\eps_k\}_{k \in \N_0} \in l^2(\N_0)$ with $\|\{\eps_k\}_k\|_{l^2} \approx \|w\|_{H^{1-s-\alpha t}}$. Using now H\"older's inequality, we obtain \begin{equation*}
\|S_{m+1}w\|_\Lt \leq C \left( \sum\limits_{k \geq 0} 2^{-2(1-s-\alpha t)} \right)^\frac{1}{2} \|w\|_{H^{1-s-\alpha t}} \leq \frac{C}{1-s} \|w\|_{H^{1-s-\alpha t}},
\end{equation*} where we used the summation formula for the geometric sum as well as the assumption that $t \in [0,\frac{7}{8}\sigma]$, $\sigma : = \frac{1-s}{\alpha}$. Consequently, \begin{equation*}
\|\partial_{x_j}([\Dn,S_{m-1}a]S_{m+1}(\partial_{x_h}w))\|_\Lt \leq \frac{C}{1-s} \|a\|_{\Lip}\|w\|_{H^{1-s-\alpha t}}
\end{equation*} and, \begin{eqnarray}
\label{est4} && \sum\limits_{\nu \geq 0} 2^{-2(s+\alpha t)\nu} \big\|\partial_{x_j}[\Dn,S_{k-3}a]\Delta_k(\partial_{x_h}w) \big\|_\Lt^2 \\
\nonumber && \quad = \sum\limits_{0 \leq \nu \leq m+3} 2^{-2(s+\alpha t)\nu} \big\|\partial_{x_j}[\Dn,S_{k-3}a]\Delta_k(\partial_{x_h}w) \big\|_\Lt^2 \leq \frac{C_{m}}{(1-s)^2} \|a\|_{\Lip}^2 \|w\|_{H^{1-s-\alpha t}}^2.
\end{eqnarray} Now, we consider \begin{equation*}
\partial_{x_j}\Big( \sum\limits_{k \geq m+2} [\Dn, S_{k-3}a]\Delta_k (\partial_{x_h}w) \Big) = \partial_{x_j}\Big( \sum\limits_{k \geq m+2} [\Dn, S_{k-3}a]\partial_{x_h}(\Delta_k w) \Big).
\end{equation*} Looking at $\spec([\Dn, S_{k-3}a]\Delta_k (\partial_{x_h}w))$, we see that $[\Dn, S_{k-3}a]\Delta_k (\partial_{x_h}w)$ is identically $0$ if $|k-\nu| \geq 4$. This means that the sum runs over at most seven terms: from $\partial_{x_j}[\Dn, S_{\nu-6}a]\partial_{x_h}(\Delta_{\nu-3}w)$ up to $\partial_{x_j}[\Dn, S_{\nu}a]\partial_{x_h}(\Delta_{\nu+3}w)$, where each of them has a spectrum contained in a ball $\{|\xi| \leq C 2^{\nu}\}$. We consider only one of these terms, e.g. $\partial_{x_j}[\Dn, S_{\nu-6}a]\partial_{x_h}(\Delta_{\nu-3}w)$ since the estimates for the others follow analogously. From Bernstein's inequality we get \begin{equation*}
\|\partial_{x_j}[\Dn, S_{\nu-6}a]\partial_{x_h}(\Delta_{\nu-3}w)\|_\Lt \leq C 2^{\nu}\|[\Dn, S_{\nu-6}a]\partial_{x_h}(\Delta_{\nu-3}w)\|_\Lt
\end{equation*} and, using again \eqref{eq:CM}, \begin{equation*}
\|[\Dn, S_{\nu-6}a]\partial_{x_h}(\Delta_{\nu-3}w)\|_\Lt \leq C\|a\|_{\Lip} \|\Dn w\|_\Lt.
\end{equation*} Hence, we have \begin{equation*}
\|\partial_{x_j}[\Dn, S_{\nu-6}a]\partial_{x_h}(\Delta_{\nu-3}w)\|_\Lt \leq C 2^\nu \|a\|_{\Lip} \|\Dn w\|_\Lt.
\end{equation*} Thus squaring, multiplying by $2^{-2(s+\alpha t)\nu}$ and summing over $\nu$, we get \begin{eqnarray*}
&& \sum\limits_{\nu \geq 0} 2^{-2(s+\alpha t)\nu}\|\partial_{x_j}[\Dn,S_{\nu-3}]\partial_{x_h}(\Dn u)\|_\Lt^2 \\
&& \qquad \qquad \qquad \leq C \|a\|_{\Lip}^2 \sum\limits_{\nu \geq 0} 2^{2(1-s-\alpha t)}\|\Delta_\nu w\|_\Lt^2.
\end{eqnarray*} With $w \in H^{1-s-\alpha t}(\R^n_x)$ and using Proposition \ref{RevSobolevLP}, we finally get \begin{equation*}
\sum\limits_{\nu \geq 0} 2^{-2(s+\alpha t)\nu}\|\partial_{x_j}[\Dn,S_{\nu-3}]\partial_{x_h}(\Dn u)\|_\Lt^2 \leq C \|a\|_{\Lip}^2 \|w\|_{H^{1-s-\alpha t}}^2.
\end{equation*} As already mentioned, the other terms can be treated the same way. We finally get \begin{equation} \label{est5}
\sum\limits_{\nu \geq 0} 2^{-2(s+\alpha t)\nu}\Big\|\partial_{x_j}\big( \sum\limits_{k \geq m+2}[\Dn,S_{k-3}]\partial_{x_h}(\Delta_k u)\big) \Big\|_\Lt^2 \leq C \|a\|_{\Lip}^2\|u\|_{H^{1-s-\alpha t}}^2
\end{equation} which, putting \eqref{est4} and \eqref{est5} together and using the notation $v_\nu = 2^{-(s + \alpha t)\nu}w_\nu$, concludes the proof of the proposition. \end{proof}


\section*{Literature}
\begin{thebibliography}{99.}

\bibitem{AN1963}S.  Agmon and L. Nirenberg, 
\textit{Properties of solutions of ordinary differential equations in Banach spaces}, Comm. Pure  Appl. Math. {\bf 16} (1963),  121--239.

\bibitem{CL1995} F. Colombini and N. Lerner, \textit{Hyperbolic operators having non-Lipschitz coefficients}, Duke Math. J.  \textbf{77} (1995), no. 3,  657--698.

\bibitem{CM1978} R. Coifman and Y. Meyer, 
\textit{Au-del\`a des op\'erateurs pseudo-diff\'erentiel}, Ast\'erisque {\bf 57}, Soci\'et\'e Math\'ematique de France, Paris, 1978

\bibitem{CM2008} F. Colombini and G. M\'etivier, 
\textit{The Cauchy problem for wave equations with non-Lipschitz coefficients; application to unique continuation of solutions of some nonlinear wave equations}, Ann. Sci. \'Ec. Norm. Sup\' er. (4) \textbf{41} (2008), no. 2, 177--220.



\bibitem{DSP2005} D. Del Santo and M. Prizzi,
\textit{Backward uniqueness for parabolic operators whose coefficients are non-Lipschitz continuous in time}, J. Math. Pures Appl. (9)  \textbf{84} (2005), no. 4, 471--491.

\bibitem{DSP2009} D.  Del Santo and M. Prizzi, 
\textit{Continuous dependence for backward parabolic operators with Log-Lipschitz coefficients}, Math. Ann. \textbf{345} (2009), no. 1, 213--243.

\bibitem{DSP2012} D.  Del Santo and M. Prizzi, 
\textit{A new result on backward uniqueness for parabolic operators}, Ann. Mat. Pura Appl. (4) (2014), doi: 10.1007/s10231-013-0381-3.

\bibitem{DSJP2013} D. Del Santo, Ch. P. J\"ah and M. Paicu, 
\textit{Backward-uniqueness for parabolic operators with non-Lipschitz coefficients}, to appear in Osaka J. Math., arXiv:1404.7405. 

\bibitem{Glagoleva1963} R. Ja. Glagoleva, 
\textit{Continuous dependence on initial data of the solution to the first boundary value problem for parabolic equations with negative time}, (English. Russian original) Sov. Math., Dokl. \textbf{4} (1963), 13--17; translation from Dokl. Akad. Nauk SSSR {\bf148} (1963), 20--23.

\bibitem{Hadamard1952} J. Hadamard, 
\textit{Lectures on Cauchy's Problem in Linear Partial Differential Equations}, Dover Publications, New York, 1953.

\bibitem{Hadamard1964} J. Hadamard, 
\textit{La Th\'eorie des \'Equations aux D\'eriv\'ees Partielles}, \'Editions Scientifiques, Peking; Gauthier-Villars \'Editeur, Paris, 1964.

\bibitem{Hurd1967} A. E. Hurd, 
\textit{Backward continuous dependence for mixed parabolic problems}, Duke Math. J.  \textbf{34} (1967),  493--500.



\bibitem{John1960} F. John, 
\textit{Continuous dependence on data for solutions of partial differential equations with a prescribed bound}, Comm. Pure Appl. Math.  \textbf{13} (1960), 551--585.


\bibitem{Mandache} N. Mandache, 
\textit{On a counterexample concerning unique continuation for elliptic equations in divergence form with H\"older continuous coefficients}, Math. Phys. Anal. Geom. {\bf 1}(1998), no. 3, 273--292.

\bibitem{Metivier} G. M\'etivier,
\textit{Para-differential Calculus and Applications to the Cauchy Problem for Nonlinear Systems}, Centro di Ricerca Matematica Ennio De Giorgi (CRM) Series, {\bf 5}. Edizioni della Normale, Pisa, 2008.

\bibitem{Miller1973} K.  Miller,
\textit{Nonunique continuation for uniformly parabolic and elliptic equations in self-adjoint divergence form with H\"older continuous coefficients},  Arch. Rational Mech. Anal. \textbf{54} (1974), 105--117.

\bibitem{Osgood} W. F. Osgood, 
\textit{Beweis der Existenz einer L\"osung der Differentialgleichung $\frac{dy}{dx}=f(x,y)$ ohne Hinzunahme der Cauchy-Lipschitz'schen Bedingung},  (German) Monatsh. Math. Phys. {\bf 9} (1898), no. 1, 331--345.



\bibitem{Tychonov} A. Tychonoff, 
\textit{Th\'eor\`emes d'unicit\'e pour l'\'equation de la chaleur}, Mat. Sb.  {\bf 42} (1935), no. 2, 199--215. 

\end{thebibliography}

\end{document}